\documentclass{amsart}
\usepackage{amsmath,amsthm}
\usepackage{amsfonts,amssymb}

\usepackage{enumerate}

\newtheorem{thm}{Theorem}[section]
\newtheorem{cor}[thm]{Corollary}
\newtheorem{lem}[thm]{Lemma}
\newtheorem{prop}[thm]{Proposition}

\newtheorem{defn}[thm]{Definition}

\theoremstyle{remark}
\newtheorem{rem}{Remark}[section]

 \def\a{{\alpha}} 
 \def\b{{\beta}}
 \def\g{{\gamma}}

 \def\l{{\lambda}}
 \def\d{{\delta}}
 
 \def\s{{\sigma}}
 \def\bg{{\boldsymbol{\g}}} 
 \def\one{{\mathbf 1}}
 \def\zero{{\mathbf 0}}
 
 \def\la{{\langle}}
 \def\ra{{\rangle}} 
 \def\ve{{\varepsilon}}

 \def\jb{{\mathbf j}}
 \def\kb{{\mathbf k}}
 \def\mb{{\mathbf m}}
 \def\nb{{\mathbf n}}

 \def\ub{{\mathbf u}}
 
 \def\xb{{\mathbf x}}
 \def\yb{{\mathbf y}}

 \def\CV{{\mathcal V}}
 \def\CU{{\mathcal U}}

 \def\NN{{\mathbb N}}

 \def\RR{{\mathbb R}}
  
 \def\ZZ{{\mathbb Z}}

\begin{document}
 
\title{Sobolev Orthogonal Polynomials on a Simplex}

\author{Rab\.{I}a Akta\c{S}}
\address{Departments of Mathematics\\ Ankara University\\ Ankara 06100, Turkey}
\email{raktas@science.ankara.edu.tr}

\author{Yuan Xu} 
\address{Department of Mathematics\\ University of Oregon\\
    Eugene, Oregon 97403-1222.}\email{yuan@math.uoregon.edu}
\thanks{The work of the second author was supported in part by NSF Grant DMS-1106113}

\date{\today}

\begin{abstract}
The Jacobi polynomials on the simplex are orthogonal polynomials with respect to the weight 
function $W_\bg(x) = x_1^{\g_1} \cdots x_d^{\g_d} (1- |x|)^{\g_{d+1}}$ when all $\g_i > -1$ and 
they are eigenfunctions of a second order partial differential operator $L_\bg$. 
The singular cases that some, or all, $\g_1,\ldots,\g_{d+1}$ are $-1$ are studied in this paper. 
Firstly a complete basis of polynomials that are eigenfunctions of $L_\bg$ in each singular case is found. Secondly, these polynomials are shown to be orthogonal with respect to an inner product which is explicitly determined. This inner product involves derivatives of the functions, hence the name Sobolev orthogonal polynomials. 
\end{abstract}

\maketitle

\section{Introduction}
\setcounter{equation}{0}   

The purpose of this paper is to study the limiting case of classical orthogonal polynomials 
on the simplex when the weight function becomes singular. Let $T^d$ be the $d$-dimensional
simplex defined by 
$$
  T^d := \{x \in \RR^d: x_1 \ge 0, \ldots, x_d \ge 0, 1-|x| \ge 0 \},  
$$
where $|x| : =x_1+\cdots + x_d$. The classical weight function on $T^d$ is defined by 
\begin{equation}\label{Weight}
    W_\bg (x): = x_1^{\g_1} \cdots x_d^{\g_d} (1- |x|)^{\g_{d+1}}, \quad x \in T^d,  
\end{equation}
where $\g_i$ are real numbers, usually assumed to satisfy $\g_i > -1$ to ensure the integrability
of $W_\bg$ on $T^d$. Let $c_\bg = 1 \Big / \int_{T^d} W_\bg(x) dx$ denote the normalization 
constant.

Let $\Pi^d=\RR[x]$ be the ring of polynomials in $d$-variables and let $\Pi_{n}^{d}$ denote the 
subspace of polynomials of total degree at most $n$.  When all $\g_i > -1$, the bilinear form 
$$
       \la f, g\ra_\bg : = c_\bg \int_{T^d} f(x) g(x) W_\bg(x) dx
$$
defines an inner product on the space $\Pi^d$ of polynomials of $d$-variables. The orthogonal polynomials 
with respect to this inner product have been studied extensively (cf. \cite{DX}). In particular, let 
$\CV_n^d(W_\bg)$ denote the space of orthogonal polynomials of degree $n$ with respect to 
$\la \cdot, \cdot \ra_\bg$. Then, for all $P  \in \CV_n^d(W_\bg)$,
\begin{align} \label{diff-eqn}
 L_\bg P : = & \sum_{i=1}^d x_i(1 -x_i)   \frac {\partial^2 P} {\partial x_i^2}   - 
 2 \sum_{1 \le i < j \le d} x_i x_j \frac {\partial^2 P}{\partial x_i 
 \partial x_j}  \\ 
&\,  + \sum_{i=1}^d \left( \g_i +1-(|\bg|+d+1) x_i \right) \frac {\partial P}{\partial x_i} 
 =  -n  \left(n+|\bg| +d \right)P,  
\notag
\end{align} 
where $|\bg|:= \g_1 + \cdots + \g_{d+1}$.  In other words, orthogonal polynomials of degree $n$ are
eigenfunctions of the differential operator $L_\bg$.  

If some or all $\g_i$ equal to $-1$, the weight function becomes singular and orthogonal polynomials
with respect to $W_\bg$ are no longer well defined. In fact, the Jacobi polynomials on the simplex 
defined via the Rodrigue formula can be extended to the case of $\g_i \le -1$. In the case of all
$\g_i = -1$, they have been used to study certain approximation processes in \cite{Sa, W}, but this
family of polynomials does not provide a complete basis. On the other hand, the equation 
\eqref{diff-eqn} still makes sense if some, or all, $\g_i = -1$. Thus, it is natural to ask two questions: 
firstly if \eqref{diff-eqn} still has a complete basis of polynomial solutions, and secondly, when the answer
to the first question is affirmative, if these polynomial solutions are orthogonal with respect to an inner 
product. 

The main results of this paper answer both questions affirmatively. More precisely, for each 
singular case of some or all $\g_i = -1$, we have identified a complete basis of polynomials that 
are eigenfunctions of $L_\bg$, which are given explicitly, and we have found an inner product 
explicitly, with respect to which these polynomials are orthogonal.  The inner product turns out to 
involve derivatives of the functions, so that the orthogonality is not established in the usual 
$L^2(W_\bg)$ sense, but in a Sobolev space. Such orthogonal polynomials are named Sobolev 
orthogonal polynomials. 

The problems are motivated by a recent study \cite{PX, X2} of analogue problems on the unit ball, 
for which the classical weight function takes the form $W_\mu(x) = (1-\|x\|^2)^\mu$.  The orthogonal
polynomials with respect to $W_\mu$ on the ball are eigenfunctions of a second order differential 
operator $L_\mu$. In \cite{X2}, a family of orthogonal polynomials with respect to an inner product 
that involves the first partial derivatives on the unit ball was studied, which turned out to be the 
eigenfunctions of $L_{-1}$ as shown in \cite{PX}. 

The Sobolev orthogonal polynomials have been studied extensively when $d =1$ (cf. \cite{G}). 
The simplex $T^d$ becomes, when $d =1$, the interval $[0,1]$ and $W_\bg$ is the Jacobi weight 
function $x^{\a}(1-x)^{\b}$, where we have written $\a = \g_1$ and $\b = \g_2$. In this case, Sobolev 
orthogonal polynomials have been studied by several authors  (\cite{AFR, AMR,APP, KL}). More 
precisely, the Sobolev orthogonality of the Jacobi polynomials $\{P_n^{(-N,\b)}\}_{n \ge 0}$ was 
studied in \cite{APPR}, and the case $\{P_n^{(-N, -N)}\}_{n \ge 0}$ was studied in \cite{APP}, in 
both cases $N$ is a positive integer. The particular case of the Jacobi polynomials 
$\{ P_{n}^{( -1,-1)}\} _{n\ge 0}$ had been previously given in \cite{KL}. In the case of $d =2$, the 
simplex becomes a triangle and the weight function is usually denoted by $W_{\a,\b,\g}$. An 
observation in \cite{BDFP} shows that the monic basis for $W_{\a,\b,\g}$ is still a basis when 
$\g = -1$ and  $\a, \b > -1$, and they are orthogonal with respect to an inner product that involves 
the first order derivatives. In contrast to one variable, there are only a handful papers on Sobolev
orthogonal polynomials of several variables \cite{BDFP, PI, LL, PX, X1,X2}, see also \cite{FMPP}. 
It should be mentioned that the studies in \cite{X1,X2} are motivated by problems in numerical 
solution of Poisson equations \cite{AH} and in optics. Given this background, it is somewhat 
surprising that there have been so few studies of the Sobolev orthogonal polynomials of several 
variables. It is our hope that the results in this paper will help to kindle more interests on this topic. 

The paper is organized as follows. In the next section we recall results on Jacobi polynomials 
and classical orthogonal polynomials on the simplex, where we will also present several new 
properties on the simplex, including how orthogonal bases on the faces of $T^d$ arise 
from some of the orthogonal polynomials on $T^d$, which are of independent interest.
The main results are stated in the third section, where we will state the results for $d =2$ first to 
illustrate the results and illuminate how the results are obtained, as the results in $T^d$ require 
unavoidably heavy notations. Finally, the proofs of the main results, including several lemmas, 
are given in the fourth section. 

\section{Orthogonal polynomials on the simplex}
\setcounter{equation}{0}
   
This section contains background results on orthogonal polynomials that are necessary for 
latter sections.  After a brief subsection on the classical Jacobi polynomials, orthogonal 
polynomials on the triangle are described in the second subsection and those on the simplex 
are developed in full generality in the third subsection, and orthogonal polynomials on the 
faces of the simplex are discussed in the fourth subsection. 

\subsection{Jacobi polynomials}
The Jacobi weight function $v_{\a,\b}$ is defined by 
$$
  v_{\a,\b}(x): = (1-x)^\a (1+x)^\b,  \qquad \a,\b > -1.  
$$
The Jacobi polynomials are given explicitly by the Rodrigue formula
\begin{equation}\label{Jacobi-Rodrigue}
   P_{n}^{(\a,\b)} (x) = \frac{(-1)^{n}}{2^nn!} (1-x)^{-\a} (1+x)^{-\b} \frac{d^n} {d x^n} \left[(1-x)^{n+\a} (1+x)^{n+\b}\right]
\end{equation}
and they satisfy the orthogonality condition
\begin{align} \label{Jacobi-ortho}
  c_{\a,\b}  \int_{-1}^1 P_n^{(\a,\b)} (x) P_m^{(\a,\b)} (x) v_{\a,\b}(x) dx 
      =  h_n^{(\a,\b)} \delta_{n,m}, 
\end{align}
where 
\begin{align*}
       c_{\a,\b}: = \frac{\Gamma(\a+\b+2)}{2^{\a+\b+1} \Gamma(\a+1)\Gamma(\b+1)}, \quad
       h_n^{(\a,\b)}:= \frac{(\a+1)_n (\b+1)_n (\a+\b+n+1)}{n!(\a+ \b+2)_n (\a+\b+ 2n+1)}
\end{align*}
and $\delta_{n,m}$ is the Kronecker delta. The Jacobi polynomial $P_n^{(\a,\b)}$ also satisfies
the second order differential equation 
\begin{equation}\label{Jacobi-ode}
       (1-x^2)y''+\left[\b-\a-(\a+\b+2)x\right]y' = -n(\a+\b+n+1)y, 
\end{equation}
in other words, it is the eigenfunction of the differential operator in the left hand side with
the eigenvalue $-n(\a+\b+n+1)$. 

If one of $\a$ and $\b$ is a negative integer, then the weight function $v_{\a,\b}$ is no longer
integrable on $[-1,1]$. Assume $\a = - l$, $l \in \NN$, then we have \cite[(4.22.2)]{S}
\begin{equation}\label{Jacobi-negative}
  \binom{n}{l} P_n^{(-l,\b)}(x) = \binom{n+\b}{l} \left(\frac{x-1}{2}\right)^l P_{n-l}^{(l,\b)}(x),
\end{equation}
which is well defined for $n \ge l$. In the case of $\a=-1$, defining $P_0^{(-1,\b)} (x) =1$, then 
$\{P_n^{(-1,\b)}\}_{n \ge 0}$ is well defined and they still satisfy the differential equation \eqref{Jacobi-ode}
for all $n = 0,1,\ldots$. They are, however, no longer orthogonal polynomials in the sense of 
\eqref{Jacobi-ortho} but, as it turns out, orthogonal polynomials with respect to
the following Sobolev type inner product \cite {APPR, KL},
$$
  \la f, g \ra:= {\l}f(1)g(1)+ \int_{-1}^{1} (x+1)^{\b+1}f'(x) g'(x)dx, \quad \l > 0, \quad \b > -1.
$$
If both $\a = \b=-1$, then the Jacobi polynomials in \eqref{Jacobi-Rodrigue} are well defined for
$n \ge 2$. Furthermore, any linear polynomial will be a solution of the equation \eqref{Jacobi-ode}
when $n = 0$ and $1$. In fact, the polynomials $P_{0}^{(-1,-1)} (x)=1$, $P_1^{(-1,-1)} (x)= x+\mu$ 
and
\begin{equation*} 
  P_{n}^{(-1,-1)} (x)=\frac{1}{4}(x-1)(x+1)P_{n-2}^{(1,1)} (x),\quad n\ge2,
\end{equation*}
satisfy the equation \eqref{Jacobi-ode} for all $n = 0,1, 2, \ldots$. Furthermore, these polynomials 
are orthogonal with respect to the Sobolev type inner product 
$$
  \la f, g \ra := {\l_1}f(1)g(1)+  {\l_2}f(-1)g(-1)+ \int_{-1}^{1} f'(x) g'(x)dx
$$
where $\l_1$ and $\l_2$ are nonnegative numbers, not both zero, and the constant in 
$P_1^{(-1,-1)}(x)$ is given by $\mu=\frac {\l_2-\l_1} {\l_1+\l_2}$.

Since the simplex $T^d$ becomes $[0,1]$ when $d =1$, it is more convenient to consider 
the Jacobi polynomials on the interval $[0,1]$, which we denote by $J_n^{(\a,\b)}$, and 
normalized as 
$$
   J_n^{(\a,\b)}(x) := (1-x)^{-\a} x^{-\b} \frac{d^n} {d x^n} \left[(1-x)^{n+\a} x^{n+\b}\right].
$$
Up to a constant multiple, $J_n^{(\a,\b)}(x) = c P_n^{(\a,\b)}(2x-1)$. These polynomials are 
orthogonal on $[0,1]$ with respect to the weight function
$$
    w_{a,\b}(x) := (1-x)^\a x^\b, \quad  \a, \b > -1
$$
and they satisfy a second order differential equation
$$
        x(1-x)y''+\left[\b+1 - (\a+\b+2)x\right] y' = -  n( n + \a+\b +1)y 
$$

\subsection{Orthogonal polynomials on the triangle}
On the triangle $T^2 := \{(x,y) \in \RR^2: x,y \ge 0, 1-x-y \ge 0\}$, the Jacobi polynomials are orthogonal with respect to the weight function 
$$
  W_{\a,\b,\g}(x,y):= x^{\a} y^{\b} (1- x-y)^{\g},\quad \a,\b,\g> -1.
$$
More precisely, we define the inner product 
$$
        \la f, g \ra_{\a,\b,\g} := c_{\a,\b,\g} \int_{T^2} f(x,y) g(x,y) W_{\a,\b,\g} (x,y) dxdy, 
$$
where $c_{\a,\b,\g}$ is the normalization constant of $W_{\a,\b,\g}$ given by
$$
   c_{\a,\b,\g} := 1 \Big / \int_{T^2} W_{\a,\b,\g}(x,y) dxdy = \frac{\Gamma(\a+\b+\g+ 3)}{\Gamma(\a+1) \Gamma(\b+1) \Gamma(\g+1)}. 
$$
Let $\CV_n^2(W_{\a,\b,\g})$ denote the space of orthogonal polynomials of degree $n$ with respect 
to this inner product. Then $P \in \CV_n^2(W_{\a,\b,\g})$ if $\la P, q \ra_{\a,\b,\g} = 0$ for all $q \in \Pi_{n-1}^2$
and  $\dim \CV_n^2(W_{\a,\b,\g}) = (n+1)$. Among many bases for $\CV_n^2(W_{\a,\b,\g})$, one is 
given by the Rodrigue formula: For $0 \le k \le n$,  
\begin{align} \label{Rodrigue d=2}
P_{k,n}^{(\a,\b,\g)} (x,y) := & \left[ W_{\a,\b,\g}(x,y)\right]^{-1} \frac{\partial^{n}} {\partial x^{k} \partial y^{n-k}} \left[x^{\a + k} y^{\b + n-k} (1-x-y)^{\g+n} \right]. 
\end{align}
It should be noted that the elements of this basis are  not mutually orthogonal to each other. 
The triangle $T^2$ is symmetric under the permutation of $(x,y,1-x-y)$. Parametrizing the triangle 
differently leads to two more orthogonal bases of $\CV_n^2(W_{\a,\b,\g})$. Indeed, if we define 
\begin{equation} \label{bases d=2}
Q_{k} ^{n} (x,y):=P_{k,n}^{(\b,\a,\g)} (y,x)\quad \textit{and} \quad R_{k} ^{n} (x,y):=P_{k,n}^{(\g,\b,\a)} (1-x-y,y), 
\end{equation}
then both $\{R_{k} ^{n}: 0 \le k \le n\}$ and  $\{Q_{k} ^{n}: 0 \le k \le n\}$ are bases of 
$\CV_n^2(W_{\a,\b,\g})$. Upon changing variables from \eqref{Rodrigue d=2}, these polynomials 
are given explicitly by 
$$
Q_{k} ^{n} (x,y)=\left[ W_{\a,\b,\g}(x,y)\right]^{-1} 
     \frac{\partial^{n}} {\partial x^{n-k} \partial y^{k}} \left[x^{\a + n-k} y^{\b + k} (1-x-y)^{\g+n} \right]
$$ 
and 
\begin{align*}
R_{k} ^{n} (x,y) = & \, (-1)^{k} \left[ W_{\a,\b,\g}(x,y)\right]^{-1}  \frac{\partial^{n}} {\partial x^{k}(\partial y-\partial x)^{n-k}} \left[ y^{\b + n-k} (1-x-y)^{\g+k} x^{\a + n} \right].
\end{align*}  


For the study of the Sobolev orthogonal polynomials, it is necessary to understand the restriction 
of the basis element on the boundary of the triangle. For later use, we state the following proposition,
which is the special case of Lemma \ref{lem:H_n^d}.

\begin{prop} \label{prop:2.1}
For $n=0,1,\ldots,$
\begin{enumerate}[$1.$]
\item The restriction of  $P_{0,n}^{(\a,\b,\g)} \in \CV_n^2(W_{\a,\b,\g})$ 
   on the line $x=0$ is $J_{n}^{(\g,\b)} (y)$.
\item The restriction of $P_{0,n}^{(\b,\a,\g)} \in \CV_n^2(W_{\a,\b,\g})$ 
   on the line $y=0$ is $J_{n}^{(\g,\a)} (x)$.
\item The restriction of $P_{0,n}^{(\g,\b,\a)} \in \CV_n^2(W_{\a,\b,\g})$
   on the line $x+y=1$ is $J_{n}^{(\a,\b)} (y)$.
\end{enumerate}
\end{prop}

There is another basis $\{V_{k,n}^{(\a,\b,\g)} : 0 \le k \le n \}$ of $\CV_n^2(W_{\a,\b,\g})$ 
so that $P_{k,n}^{(\a,\b,\g)}$ and $V_{k,n}^{(\a,\b,\g)}$ are biorthogonal. This basis will 
be the special case of $d =2$ of the basis \eqref{Monomial} in the following section. 

\subsection{Orthogonal polynomials on the simplex} To simplify the notation, we set, for 
$x \in T^d$, $|x| : = x_1+\cdots + x_d$ and $x_{d+1} := 1- |x|$. Furthermore, for  $\g_i  > -1$, 
$1 \le i \le d+1$, we set 
$$
  \bg := (\g_1,\ldots, \g_d, \g_{d+1}) = (\g, \g_{d+1}) \quad\textit{with}\quad  \g := (\g_1,\ldots, \g_d). 
$$
Then the weight function $W_\bg$ in \eqref{Weight} can be written as 
\begin{equation} \label{eq:weight}
   W_\bg(x) = x^\g (1-|x|)^{\g_{d+1}}.
\end{equation}
If all $\g_i > -1$, then $W_\bg$ is integrable and we can consider orthogonal polynomials
in the space $L^2(W_\bg, T^d)$ with respect to the inner product 
\begin{equation} \label{eq:ipd}
        \la f, g \ra_\bg := c_\bg \int_{T^d} f(x) g(x) W_\bg (x) dx, 
\end{equation}
where $c_\bg$ is the normalization constant of $W_\bg$ given by
$$
   c_\bg := 1 \Big / \int_{T^d} W_\bg(x) dx = \frac{\Gamma(|\bg|+ d+1)}{\prod_{i=1}^{d+1} \Gamma(\g_i+1)}. 
$$
Let $\CV_n^d(W_\bg)$ denote the space of orthogonal polynomials of degree $n$ with respect 
to this inner product. As in the case of two variables, there are several distinguished bases of 
$\CV_n^d(W_\bg)$. First we state the basis defined via the Rodrigue formula. 

\begin{lem}
For $\nb \in \NN_0^d$, $\bg \in \RR^{d+1}$ and $x \in \RR^d$, define
\begin{equation} \label{Rodrigue}
 P^{\bg}_\nb (x) := x^{-\g} (1-|x|)^{-\g_{d+1}} \frac{\partial^{|\nb|}} {\partial x^{\nb}} 
     \left[x^{\g + \nb} (1-|x|)^{\g_{d+1} + |\nb|} \right],
\end{equation}
where $\frac{\partial^{|\nb|}} {\partial x^{\nb}}=\frac{\partial^{|\nb|}} {\partial x_1^{n_1} \cdots 
\partial x_d^{n_d}}$.
If $\g_i > -1$, $1 \le i \le d+1$, then $P_\nb^{\bg}$ are orthogonal polynomials with respect to $W_\bg$ and 
$\{P_{\nb}^{\bg}: |\nb| = n\}$ is a basis of $\CV_n^d(W_\bg)$. 
\end{lem}

This is a classical result, see, for example, \cite[p. 49]{DX}. As a basis of $\CV_n^d(W_\bg)$, 
these polynomials
are eigenfunctions of the differential operator $L_\bg$ in \eqref{diff-eqn}; that is,
\begin{equation}\label{diff-eqn2}
       L_\bg P_{\nb}^\bg = - n (n+|\bg|+d)P_{\nb}^\bg, \qquad |\nb| =n.
\end{equation}
In particular, the differential equation has a complete set of solutions consisting of 
$\binom{n+d-1}{n}$ linearly independent polynomials of degree $n$. 

\begin{rem} \label{remark:g<-1}
Analytic continuation 
shows that \eqref{diff-eqn2} still holds if some or all $\g_i$ are $\le -1$, in which case, however,
$\{P_\nb^\bg: |\nb| = n\}$ no longer contains a complete set of solutions. 
\end{rem}

If some components of $\bg$ are negative integers, the polynomial $P^{\bg}_\nb$ still satisfies 
some orthogonality with respect to a different weight function. More precisely, we have the 
following lemma.

\begin{lem}\label{orthogo.}
Let $k$ be an integer, $1 \le k \le d$, and $\g \in \RR^{d-k+1}$ with $\g_i>-1$ for all $i$. 
Let $\bg = (\g, - \mb)$, 
where $\mb \in \NN^k$. Then $P_\nb^{\bg}$ is orthogonal to polynomials of degree at most $|\nb| - |\mb|-1$
with respect to $W_{\g, \zero_k}$; that is, 
$$
\int_{T^d}P_\nb^{(\g,-\mb)}(x) Q(x)W_{\g,\zero_k}(x) dx=0,  \qquad    \forall Q\in\Pi_{|\nb|-|\mb|-1}^d.
$$
\end{lem}

\begin{proof}
Let us denote temporarily $x = (\xb_1,\xb_2)$ with $\xb_1 \in \RR^{d-k+1}$ and
$\xb_2 \in \RR^{k-1}$ and, similarly, $\nb = (\nb_1, \nb_2)$. Furthermore, let $\mb = (\mb', m_0)$ with 
$\mb' \in \NN^{k-1}$. Then, by the Rodrigue formula \eqref{Rodrigue},
\begin{align*}
& J: =  \int_{T^d}P_\nb^{(\g,-\mb)}(x) Q(x)W_{\g,\zero_k}(x)dx \\
 & \qquad \quad  =
    \int_{T^d} \xb_2^{\mb'}(1-|x|)^{m_0}\frac{\partial^{|\nb|}}{\partial x^\nb}
       \left\{\xb_1^{\nb_1+\g} \xb_2^{\nb_2 - \mb'} (1-|x|)^{|\nb|-m_0}\right\} Q(x)dx.
\end{align*}
We perform integration by parts on the last integral. Since $\g_i > -1$, the integration by parts over the 
first $d-k+1$ variables carries through without problem. For integration by parts in $x_{d-k+2}$, the 
first component of $\xb_2$, the part that is being integrated out, is zero since, if $j \le m_1$, then 
$\frac{\partial^{j-1}}{\partial x_{d-k+2}^{j-1}} [\xb_2^{\mb'}(1-|x|)^{m_0}Q(x) ]$ vanishes for $x_{d-k+2} =0$ 
and $|x| =1$, whereas if $j >m_1$, then for $\ell = n_{d-k+2} - j$, $\frac{\partial^\ell}{\partial x_{d-k+2}^{\ell}}
\left\{\xb_2^{\nb_2 - \mb'} (1-|x|)^{|\nb|-m_0}\right\}$ vanishes on the boundary. Consequently, we 
conclude that 
$$
   J =  (-1)^{|\nb|} \int_{T^d}   \xb_1^{\nb_1+\g} \xb_2^{\nb_2 - \mb'}
        (1-|x|)^{|\nb|-m_0}  \frac{\partial^{|\nb|}}{\partial x^\nb} \left[ \xb_2^{\mb'}(1-|x|)^{m_0}Q(x)\right]dx =0
$$
since the term insider the bracket is a polynomial of degree $|\mb| + \deg Q < n$. 
\end{proof}

It should be noted that, instead of $\bg = (\g, -\mb)$, we can assume that $\g_i$ is a negative integer
for $i$ in a subset of $\{1,2,\ldots, d+1\}$ in the above lemma. The result will be similar and the 
precise formulation should be clear from the above proof. 

The simplex $T^d$ is symmetric under the permutation of $(x_1,\ldots,x_d,x_{d+1})$. This symmetry
carries over to the Rodrigue basis. Let 
$$
        \ZZ_{d} := \{1,2,\ldots, d\}.
 $$ For a subset $S$ of $\ZZ_{d}$
we denote by $|S|$ the cardinality of $S$ and by $S^c$ the complement of $S$, that is,
$S^c:= \ZZ_{d} \setminus S$.

Let $\mathcal{G}_k$ be the permutation group of $k$ elements; its action on a function 
$g: \RR^k \mapsto \RR$ is denoted by $g(x \s)$, $\s  \in \mathcal{G}_k$.  Just as in the 
case of $d=2$, see \eqref{bases d=2}, we can permute variables and parameters of 
$P_\nb^\bg$ to obtain different bases of $\CV_n^d(W_\bg)$. 

\begin{defn} \label{defn: P(x_S)}
Let $S=\{i_1,\ldots,i_d\}$ be a subset of $\ZZ_{d+1}$ and $S^c = \{i_{d+1}\}$. 
Let $\bg_S:=(\g_S, \g_{i_{d+1}})$, where $\g_S:=(\g_{i_1},\ldots,\g_{i_d})$, and 
$x_S = (x_{i_1},\ldots,x_{i_d})$.  
\begin{enumerate}[1.]
\item If $d+1 \notin S$, then $S^c=\{d+1\}$, $x_S=(x_1,\ldots,x_d)\s$ and 
$\g_S=(\g_1,\dots, \g_d)\s$ for some $\s \in \mathcal{G}_d$, and $\bg_S=(\g_S, \g_{d+1})$. 
We denote 
$$
   P_\nb^{\bg_S}(x_S)=x_{S}^{-\g_S}(1-|x_S|)^{-\g_{d+1}}\frac{\partial^{|\nb|}}{\partial x_S^\nb}
      \left\{x_{S}^{\g_S+\nb}(1-|x_S|)^{\g_{d+1}+|\nb|}\right\}.
$$
\item If ${d+1}\in S$, then $S= \{d+1, i_2,\ldots, i_d\}\s$ for some $\s \in \mathcal{G}_d$ 
and $S^c = \{i_{d+1}\}$, where  $i_j \in \ZZ_d$. 
Let $x_S=(x_{d+1},x_{i_2},\dots,x_{i_{d}})\s$ and $\g_S=(\g_{d+1},\g_{i_2},\dots, \g_{i_{d}})\s$.
We define then
$$
  P_\nb^{\bg_S}(x_S)=  x_{S}^{-\g_S}x_{i_{d+1}}^{-\g_{i_{d+1}}}
    \frac{\partial^{|\nb|}}{\partial \xb_{S,\sigma}^\nb}
      \left\{x_{S}^{\g_S+\nb}x_{i_{d+1}}^{\g_{i_{d+1}}+|\nb|}\right\},
$$
where $\partial \xb_{S,\s} =(- \partial_{i_{d+1}}, \partial_{i_2, i_{d+1}}, \ldots, \partial_{i_d, i_{d+1}}) \s$
with  $\partial_i:=\partial{x_i}$ and $\partial_{i,j}:=\partial_i-\partial_j$.
\end{enumerate}
\end{defn}

In the case of $d+1 \in S$, the notation $P_\nb^{\bg_S}(x_S)$ agrees with that of
\eqref{Rodrigue}, and it is a simple permutation of $P_\nb^\bg(x)$. For $d =2$, see
\eqref{bases d=2}.

\begin{lem}\label{d>2 bases}
Let $S$ and $P_\nb^{\bg_S}(x_S)$ be defined in the Definition \ref{defn: P(x_S)}. 
Then $\{P_\nb^{\bg_S}(x_S): |\nb| = n\}$ is a basis of $\CV_n^d(W_{\bg})$.
In particular, $P_\nb^{\bg_S}(x_S)$ are eigenfunctions of the differential operator 
$L_\bg$ given by \eqref{diff-eqn}.
\end{lem}

\begin{proof}
In the first case, $d +1 \notin S$, changing variable $x \mapsto x \s$ in the integral 
of $\la P_\nb, g\ra_\bg$ shows immediately that $P_\nb^{\bg_S}(x_S) \in \CV_n^d(W_\bg)$.
 
The second case, $d+1 \in S$, requires a bit more work. Up to a permutation in 
$\mathcal{G}_d$, we can assume that $S = (d+1,1,\ldots, d-1)$. Then $S^c=\{d\}$, $x_S=
(x_{d+1},x_1,\dots,x_{d-1})$ and $\g_S = (\g_{d+1},\g_1,\ldots, \g_{d-1})$. 
Making a change of variable 
$$
  u_1 = x_{d+1} = 1- |x|,\quad u_2=x_1, \ldots, u_{d} = x_{d-1}
$$ 
so that $1 - |u| = x_d$ and, using superscript for the variable with respect to which 
the derivative is taking, 
$$
\partial^{(u)}_1=-\partial^{(x)}_d,\quad\partial^{(u)}_2=\partial^{(x)}_1-\partial^{(x)}_d,
 \ldots, \partial^{(u)}_{d}=\partial^{(x)}_{d-1}-\partial^{(x)}_d, 
$$
we obtain from the Rodrigue formula \eqref{Rodrigue} that  
$$
P_{\nb}^{\bg_S}(u)=(-1)^{n_1}x_S^{-\g_S} x_d^{-\g_d}\frac{\partial^{|\nb|}} {\partial_{d}^{n_1}\partial_{1,d}^{n_2}\cdots \partial_{d-1,d}^{n_{d}}} \left[x_S^{\g_S + \nb} x_d^{\g_{d} + |\nb|} \right]
= P_{\nb}^{\bg_S}(x_S).
$$
Furthermore, under this change of variables, $W_\bg(x)= W_{\bg_S}(u)$. Consequently, 
\begin{align*}
  \int_{T^d} P_{\nb}^{\bg_S}(x_S) g(x) W_\bg(x) dx = 
      \int_{T^d} P_{\nb}^{\bg_S}( u) g(u_2,\ldots,u_{d}, 1-|u|) W_{\bg_S}(u) d u = 0 
\end{align*}
by the orthogonality of $P_\nb^\bg$, so that $ P_{\nb}^{\bg_S}(x_S) \in \CV_n^d(W_\bg)$. 
\end{proof}

Another basis of $\CV_n^d(W_\bg)$ is $\{ V_\nb ^{\bg} (x): |\nb|=n, \nb \in \NN_0^d \}$, where 
$V_\nb^\bg$ are monomial orthogonal polynomials defined by 
\begin{equation} \label{Monomial}
V_\nb ^{\bg} (x) := \sum _{\mb \le \nb} (-1)^{n+|\mb|} {\prod_{i=1}^{d} \binom{n_i}{m_i} 
   \frac{ (\g_i+1)_{n_i} (|\bg|+d)_{n+|\mb|}}{(\g_i+1)_{m_i} (|\bg|+d)_{n+|\nb|}}} x^{\mb}.
\end{equation}
These polynomials satisfy $V_\nb^\bg(x) = x^\nb + Q_\nb(x)$, $Q_\nb \in \Pi_{n-1}^d$ and $V_\nb^\bg  \in 
\CV_n^d(W_\bg)$, so that $V_\nb^\bg$ is the orthogonal projection of $x^\nb$ in $\CV_n^d(W_\bg)$. 
Furthermore, we have the biorthogonal relation 
$$
\int_{T^d} P_\nb ^{\bg} (x) V_\mb ^{\bg} (x) W_\bg(x) dx =\frac{\prod_{i=1}^{d} (\g_i+1)_{n_i} 
     (\g_{d+1}+1)_{|\nb|}} {(|\bg|+d+1)_{2|\nb|}} {\nb}!  \d_{\nb,\mb}.
$$
We note that $V_\nb^\bg$ is well defined even if $\g_{d+1} = -1$, since $\g_{d+1}$ appears
only in $|\bg|$ in the right hand side of \eqref{Monomial}. 
For latter use, we state one more property of $V_\nb^{\bg}$. Let $e_1,\ldots, e_d$ be the
standard basis of $\RR^d$, that is, the $i$th coordinate of $e_j$ is 1 if $i =j$, 0 if $i \ne j$.  

\begin{lem}\label{lem:monomial-diff}
For $1 \le i \le d$, 
\begin{equation} \label{monomial-diff}
\frac{\partial}{\partial x_i}V_\nb^{\bg}(x)=n_iV_{\nb-e_i}^{(\g+e_i,\g_{d+1}+1)}(x).
\end{equation}
\end{lem}

\begin{proof}
We can assume $i =1$ and write $V_\nb^{\bg}$ as 
\begin{align*} 
V_\nb ^{\bg} (x) := \sum_{m_1=0}^{n_1} \sum _{\mb' \le \nb'} (-1)^{n+|\mb|} {\prod_{i=1}^{d} \binom{n_i}{m_i}
      \frac{ (\g_i+1)_{n_i} (|\bg|+d)_{n+|\mb|}}{(\g_i+1)_{m_i} (|\bg|+d)_{n+|\nb|}}}  {x_1}^{m_1} \xb^{\mb'}
\end{align*}
where $\xb=(x_2,\dots,x_d)$, $\nb'=(n_2,\ldots,n_d)$ and $\mb'=(m_2,\ldots,m_d)$. Taking derivative with respect to
$x_1$ and shifting the summation index over $m_1$, we obtain 
\begin{align*} 
\frac{\partial}{\partial x_1} V_\nb ^{\bg} (x) &\,=\sum_{m_1=1}^{n_1} \sum _{\mb' \le \nb'} (-1)^{n+|\mb|} {\prod_{i=1}^{d} \binom{n_i}{m_i} \frac{ (\g_i+1)_{n_i} (|\bg|+d)_{n+|\mb|}}{(\g_i+1)_{m_i} (|\bg|+d)_{n+|\nb|}}} m_1 {x_1}^{m_1-1} \xb^{\mb'}\\
&\,= \sum_{m_1=0}^{n_1-1} \sum _{\mb' \le \nb'} (-1)^{n+|\mb|-1}n_1 \binom{n_1-1}{m_1}\frac{(\g_1+1)_{n_1}}{(\g_1+1)_{m_1+1}}{x_1}^{m_1}\\
  &\, \qquad \qquad\qquad \times
  {\prod_{i=2}^{d} \binom{n_i}{m_i} \frac{ (\g_i+1)_{n_i} (|\bg|+d)_{n+|\mb|+1}}{(\g_i+1)_{m_i} (|\bg|+d)_{n+|\nb|}}}  \xb^{\mb'}.
\end{align*}
Since $(a)_n=a(a+1)_{n-1}$, it holds that
$$
\frac{(\g_1+1)_{n_1}}{(\g_1+1)_{m_1+1}}=\frac{((\g_1+1)+1)_{n_1-1}}{((\g_1+1)+1)_{m_1}}, 
\quad \frac{(|\bg|+d)_{n+|\mb|+1}}{(|\bg|+d)_{n+|\nb|}}= \frac{(|\bg|+d+2)_{n+|\mb|-1}}{(|\bg|+d+2)_{n+|\nb|-2}},
$$
from which \eqref{monomial-diff} follows. 
\end{proof}

The relation \eqref{monomial-diff} allows us to take higher order derivatives of $V_\nb^\bg$. One 
interesting consequence of the lemma is that $V_\nb^{\bg}$ are orthogonal polynomials for a family 
of other inner products that involve differentiation. To state the result, we need further notation. For 
a subset $S = \{i_1,\ldots, i_j\}$ of $\ZZ_d$, we define
$$
    \frac{\partial^{|S|}}{\partial x^S} =  \frac{\partial^{|S|}}{\partial x_{i_1} \ldots \partial x_{i_j}}.
$$

\begin{lem}\label{lem:epd}
Let $\g_i > -1$ for $1 \le i \le d+1$. For $m \le d$, define
\begin{align}\label{eq:[f,g]-general}
  [f, g ]_{\bg}: = \la f, g\ra_\bg + \sum_{j=1}^{m} \sum_{\substack { S\subset \ZZ_d  \\ |S| =j}}
     \l _S  \int_{T^d} \frac{\partial^{j}f}{\partial x^{S}} \frac{\partial^{j}g} {\partial x^{S}} 
          \prod_{i \in S}x_i (1- |x|)^{j} W_{\bg}(x) dx, 
\end{align}
where $\l _S$ are nonnegative numbers. Then the polynomials $V_\nb^{\bg}$ in \eqref{Monomial} 
are orthogonal with respect to the inner product $[\cdot,\cdot]_{\bg}$. 
\end{lem}

\begin{proof}
By \eqref{monomial-diff}, taking repeated derivatives of $V_\nb^\bg$, we have
$$
  \frac{\partial^{|S|}f}{\partial x^{S}} V_\nb^\bg(x) = n_{i_1} \cdots n_{i_{j}} 
        V_{\nb-e_{i_1}-\ldots-e_{i_{j}}}^{(\g+e_{i_1}+\ldots+e_{i_{j}},\g_{d+1}+j)}(x)
$$
with $S = \{i_1,\ldots,i_j\}$, which implies, by the orthogonality of 
$ V_{\nb-e_{i_1}-\ldots-e_{i_{j}}}^{(\g+e_{i_1}+\ldots+e_{i_{j}},\g_{d+1}+j)}$, that
$V_\nb^\bg$ is orthogonal with respect to 
$$
 [f,g]_S :=  \int_{T^d} \frac{\partial^{j}f}{\partial x^{S}} \frac{\partial^{j}g} {\partial x^{S}} 
          \prod_{i \in S}x_i (1- |x|)^{j} W_{\bg}(x) dx 
$$
and, consequently,  orthogonal with respect to $[\cdot, \cdot]_{\bg}$. Since all 
$\l_S \ge 0$, $[\cdot, \cdot]_{\bg}$ is an inner product. 
\end{proof}

\begin{rem} \label{rem:2.1}
Using the relation \eqref{monomial-diff} repeatedly, we can also include partial 
derivatives of higher orders, such as $\frac{\partial^{|\mb|}}{\partial x_1^{m_1}\cdots \partial x_k^{m_k}}$
for $\mb = (m_1,\ldots, m_k)$ with $k \le d$, in the definition of $[f,g]_\bg$. For our 
purpose, the definition in \eqref{eq:[f,g]-general} is sufficient. 
\end{rem}

A couple of remarks are in order. We note that $[\cdot,\cdot]_S$ is not an inner product by 
itself since, for example, $[f,f]_S =0$ for $f (x) =1$. Even though $\{V_\nb^\bg: |\nb| =n\}$ is 
a basis of $\CV_n^d(W_\bg)$, not every basis of $\CV_n^d(W_\bg)$ is orthogonal with 
respect to $[\cdot,\cdot]_\bg$. For example, when $\l_S$ are not all zero, $P_\nb^\bg$ 
is not an orthogonal polynomial with respect to $[\cdot,\cdot]_\bg$. 

\subsection{Orthogonal polynomials on the faces of the simplex}
In addition to the standard basis $e_1,\ldots, e_d$ of $\RR^d$, we further set $e_0 = (0,\ldots, 0)$. 
The points $e_0, e_1,\ldots, e_d$ are the vertices of the simplex $T^d$. The boundary of the 
simplex $T^d$ can be decomposed into lower dimensional faces.  

\begin{defn}
Let $S$ be a subset of $\ZZ_{d+1}$ with $1\le |S| \le d$. Then 
$$
              T_S^{d-|S|} : = \{x \in T^d: x_i =0, \,  i \in S\}
$$
is a $(d-|S|)$-dimensional face of $T^d$. 
\end{defn}

For $1\le |S| \le d$, the number of  $(d-|S|)$-dimensional faces is $\binom{d+1}{|S|}$. In particular, 
for $|S|=d $, the 0-dimensional faces are precisely the vertices $e_0, e_1, \ldots, e_{d}$ of 
$T^d$, whereas for $|S|  =1$, the $(d-1)$-dimensional faces are $T^{d-1}_{\{j\}}$, $j =1,\ldots, d, d+1$. 
Each face $T_S^{d-|S|}$ is a $({d-|S|})$-dimensional simplex and, moreover, $T_S^{d-|S|}$ is isomorphic
to the simplex $T^{d-|S|}$ in variables $\{x_i: i \in S^c\}$. Furthermore, the restriction of $W_\bg$ on $T_S^{d-|S|}$ is a weight function of the same type in variables $\{x_i: i \in S^c\}$.  

\begin{defn}\label{defn:H_n^d}
Let $S=\{i_1,\ldots,i_d\}$ be a subset of $\ZZ_{d+1}$  and $S^c = \{i_{d+1}\}$. Let 
$\bg_S:=(\g_S, \g_{i_{d+1}})$, where $\g_S:=(\g_{i_1},\ldots,\g_{i_d})$, and 
$x_S = (x_{i_1},\ldots,x_{i_d})$. Let $ P_\nb^{\bg_S}(x_S)$ be defined as Definition \ref{defn:  P(x_S)}. 
\begin{enumerate}[1.]
\item Let $d+1 \notin S$. For $S_j \subset \ZZ_d$ with $|S_j|= j$, $1 \le j \le d-1$, define
$$
   H_{n,S_j}^d(W_\bg): = \mathrm{span} \{P_{\nb}^{\bg}(x): |\nb| = n, \, \hbox{and} \, \,
           n_{\ell}= 0, \, \forall \ell \in S_j\}.
$$
\item Let $d+1 \in S$. Then $S=\{d+1,\ell_1,\ldots,\ell_{d-1}\}$ with $\{\ell_1,\ldots,\ell_{d-1}\}\subset\ZZ_d$ and $\ZZ_{d+1}\setminus S=\{\ell_d\}$.  For $S_j=\{d+1,\ell_1,\ldots,\ell_{j-1}\}$ with 
$1 \le j\le d-1$, where $S_1=\{d+1\}$, define
$$
H_{n,S_j}^d(W_\bg): = \mathrm{span} \{P_{\nb}^{(\g_S,\g_{\ell_d})}(x_{S_j},x_{S\setminus S_j}): 
   |\nb| = n \,\, \hbox{with}  \,\, n_\ell=0, \, 1 \le \ell \le j\}.
$$
\end{enumerate}
We further define $H_{n,S_d}^d(W_\bg) = \emptyset$ and $H_{n,S_{d+1}}^d(W_\bg) = \emptyset$.
\end{defn}

\begin{lem}\label{lem:H_n^d}
The space $H_{n,S_j}^d(W_\g)$ satisfies the following properties:\\ 
\noindent $(i)$ $ H_{n,S_j}^d(W_\bg)$ is a subspace of $\CV_n^d(W_\bg)$;\\
\noindent $(ii)$ The restriction of $H_{n,S_j}^d (W_\bg)$ on the face $T_{S_{j}}^{d-j}$ is the subspace 
$\CV_n^{d-j} (W_{\bg }\vert_{T_{S_{j}}^{d-j}})$, where $W_{\bg}\vert_{T_{S_{j}}^{d-j}}$ denotes the restriction of $W_\bg$ on
the face $T_{S_{j}}^{d-j}$, that is, 
$$
H_{n,S_j}^d (W_\bg) \big \vert_{T_{S_{j}}^{d-j}} = \CV_n^{d-j} \Big(W_\bg \big \vert_{T_{S_{j}}^{d-j}} \Big)
$$
in the variables $\{x_i: i \in S_{j}^c\}$. 
\end{lem}

\begin{proof}

The item (i) follows immediately from Lemma \ref{d>2 bases}. For the proof of item (ii), we need
to consider two cases, $d+1\notin S$ and $d+1 \in S$. 

\medskip
\noindent 
{\it Case 1.}  $d+1 \notin S$. Then $S= \ZZ_d$. Let $S_j=\{i_1,\ldots,i_j\}\subset\ZZ_d$, $S_j^c=\{i_{j+1},\dots,i_d\}$, $\bg_{S_j^c}=(\g_{i_{j+1}},\ldots,\g_{i_d})$ and  $\nb_ {S_j^c}=(n_{i_{j+1}},\dots, n_{i_d})$. 
 If $n_\ell=0$ for $\ell \in S_j$, then $P_\nb^{\bg}$ in \eqref{Rodrigue} can be written as 
$$
P_{\nb}^{\bg}(x)=x_{S_{j}^c}^{-\bg_{S_{j}^c}}(1-|x|)^{-\g_{d+1}}\frac{\partial^{|\nb_{S_j^c}|}} 
  {\partial x_{S_{j}^c}^{\nb_{S_{j}^c}}} 
     \left[ x_{S_{j}^c}^{\bg_{S_{j}^c} + \nb_{S_j^c}} (1-|x|)^{\g_{d+1} + |\nb_{S_j^c}|} \right],
$$
which, when restricted to the face $T_{S_{j}}^{d-j}$, becomes 
$$
P_{\nb}^{\bg}(x)\big \vert_{T_{S_{j}}^{d- j}} =x_{S_{j}^c}^{-\bg_{S_{j}^c}}(1-|x_{S_{j}^c}|)^{-\g_{d+1}}\frac{\partial^{|\nb_{S_j^c}|}} 
  {\partial x_{S_{j}^c}^{\nb_{S_j^c}}} 
     \left[ x_{S_{j}^c}^{\bg_{S_{j}^c} + \nb_{S_j^c}} (1-|x_{S_{j}^c}|)^{\g_{d+1} + |\nb_{S_j^c}|} \right],
$$
precisely the Rodrigue formula in the variable $x_{S_{j}^c}$ for the weight function 
$W_\bg \big \vert_{T_{S_{j}}^{d-j}}$ and, as such, is an element of 
$\CV_n^{d-j}\big(W_\bg \big \vert_{T_{S_{j}}^{d-j}}\big)$. Furthermore, the collection 
of such polynomials forms a basis of the latter space. This completes the proof 
of (ii) in this case. 

\medskip\noindent 
{\it Case 2.} $d+1 \in S$. Up to a permutation in $\mathcal{G}_d$, we can assume 
$S = (d+1,1,\ldots, d-1)$.
Then $S^c=\{d\}$, $\bg_S =(\g_S,\g_d)$ with $\g_S = (\g_{d+1},\g_1,\ldots, \g_{d-1})$ and $S_j=\{d+1,1,\ldots, j-1\}$. 
Making a change of variable 
$$
  u_1 = x_{d+1}=1- |x|,\quad u_2=x_1, \ldots, u_{d} = x_{d-1},
$$ 
we have $1 - |u| = x_d$, $W_\bg(x)=W_{\bg_S}(u)$ and $P_{\nb}^{\bg_S}(x_S)
=P_{\nb}^{\bg_S}(u)$, as in the proof of Lemma \ref{d>2 bases}, so that
$$
   H_{n,S_j}^d(W_\bg(x))=H_{n,S^*_{j}}^d(W_{\bg_S} (u))
$$
where $S_j^*:=\{1,2,\ldots,j\}$ and 
$$ H_{n,S_j^*}^d(W_{\bg_S}(u))
=\mathrm{span} \{P_{\nb}^{\bg_S}(u): |\nb| = n, \,\, n_\ell=0,\,\, \forall \ell \in S_j^*\}.
$$
From the  case that $d+1 \notin S$, it follows that
$$
  H_{n,S_j^*}^d (W_{\bg_S}) \big \vert_{T_{S_{j}^*}^{d-j}} = 
       \CV_n^{d-j} \Big(W_{\bg_S} \big \vert_{T_{S_{j}^*}^{d-j}} \Big),
$$
which is equivalent to 
$$
H_{n,S_j}^d (W_\bg) \big \vert_{T_{S_j}^{d-j}} = \CV_n^{d-j} \Big(W_\bg \big \vert_{T_{S_j}^{d-j}} \Big).
$$
This completes the proof. 
\end{proof}

\begin{lem}\label{rem} 
Let $H_{n,S_j}^d (W_\bg)$ be defined as in Definition \ref{defn:H_n^d}. Then $H_{n,S_j}^d (W_\bg)$
is a subspace of $\CV_n^d(W_{\bg^*})$, where $\bg^* =  (\g_1^*,\ldots,\g_{d+1}^*)$ is defined as follows: 
$\g_\ell^* \in \RR$ are arbitrary for $\ell \in S_j$ and $\g_\ell^* = \g_j$ otherwise. In particular,
every element of $H_{n,S_j}^d (W_\bg)$ satisfies the equation 
$L_{\bg^*} P = - n (n +|\bg^*| +d)P$, where $\g_\ell^*$ are arbitrary for $\ell \in S_j$. 
\end{lem}

\begin{proof}
In the Rodrigue formula $P_\nb^{\bg}$ of \eqref{Rodrigue}, if $n_\ell =0$, then the derivative 
with respect to $x_\ell$ does not appear, so that, after canceling out $x_\ell^{-\g_\ell}$ and 
$x_\ell^{\g_\ell}$, $\g_\ell$ does not appear in $P_\nb^\bg$, which means that we can assign
$\g_\ell$ any value. In other words, $P_\nb^\bg(x) = P_\nb^{\bg^*}(x)$ if $n_\ell = 0$, $\ell 
\in S_j$. Consequently, $H_{n,S_j}^d(W_\bg) \in \CV_n^d(W_\bg^*)$. 
\end{proof}

\section{Main results}
\setcounter{equation}{0}

Our main results on the simplex are fairly involved in notations. We shall state our results first 
in the case of triangle, that is, $d =2$, and use it to explain how we arrive at the results. 

\subsection{Sobolev orthogonal polynomials on the triangle}

As stated in the introduction, we are facing two problems. The first one is to find a complete 
solutions of polynomials for the differential equation \eqref{diff-eqn} in which $d =2$ and
$\bg = (\a,\b,\g)$. 

To facilitate the discussion,  we define the following subspaces of 
$\CV_n^2(W_{\a,\b,\g})$, 
\begin{align*}
 H_{n,1} (w_{\g,\b})&  : = \mathrm{span} \{P_{0,n}^{(0,\b,\g)}(x,y)\}, \quad 
 H_{n,2} (w_{\g,\a}) := \mathrm{span} \{P_{0,n}^{(0,\a,\g)}(y,x)\}, \\
  \quad  H_{n,3} (w_{\a,\b}) & : = \mathrm{span} \{P_{0,n}^{(0,\b,\a)}(1-x-y,y)\},
\end{align*}
which can be formulated as special cases of the spaces in Definition \ref{defn:H_n^d}. 
According to Proposition \ref{prop:2.1}, 
$H_{n,1}(w_{\g,\b})\vert_{x=0} \allowbreak = \CV_n^1(w_{\g,\b})$, 
$H_{n,2}(w_{\g,\a})\vert_{y=0} = \CV_n^1(w_{\g,\a})$ and 
$H_{n,3}(w_{\a,\b})\vert_{x+y=1} = \CV_n^1(w_{\a,\b})$. 
In the following we adopt the convention that 
$$
    \CV_n^2(W_{\a,\b,\g}) = \emptyset  \quad \hbox{and} \quad H_{n,j} = \emptyset, \quad \hbox{if $n < 0$}. 
$$

\begin{thm} \label{thm:main_d=2_1}
For $n \in \NN_0$ and at least one of $\a,\b,\g$ being $-1$, the differential  equation 
$$
           L_{\a, \b, \g} u =  - n(n+ \a+\b+\g+2) u 
$$
has a solution space  $\CU_n^2(W_{\a,\b,\g})$ of polynomials of degree $n$, which has
dimension $n+1$ and can be decomposed as a direct sum as follows: 
\begin{align*}
   \CU_n^2(W_{\a,\b,-1}): =  &\, (1-x-y) \CV_{n-1}^2(W_{\a,\b,1}) + H_{n,3}(w_{\a,\b}), \quad n \ge 0,\\
   \CU_n^2(W_{\a,-1,-1}):  = &\, y (1-x-y) \CV_{n-2}^2(W_{\a,1,1}) +y H_{n-1,3}(w_{\a,1}) \\
                  & \, + (1-x-y) H_{n-1,2}(w_{1,\a}) 
\end{align*}                  
for $n \ge 1$ and $\CU_0^2(W_{\a,-1,-1}) := \mathrm{span}\{1\}$, and
\begin{align*}                  
   \CU_n^2(W_{-1,-1,-1}): = & \, x y (1-x-y) \CV_{n-3}^2(W_{1,1,1}) + x y  H_{n-2,3}(w_{1,1}) \\
                  & \, + x (1-x-y) H_{n-2,2}(w_{1,1}) + y (1-x-y) H_{n-2,1}(w_{1,1})
\end{align*}   
for $n \ge 2$, $\CU_0^2(W_{-1,-1,-1}) := \mathrm{span}\{1\}$  and $\CU_1^2(W_{-1,-1,-1}) :=
    \mathrm{span}\{ x+c_1 , y+c_2\}$, where $c_1$ and $c_2$ are arbitrary constants.
\end{thm}

From the decomposition, we can deduce further relations among the subspaces.  

\begin{cor}
The subspaces in the Theorem \ref{thm:main_d=2_1} satisfy the relations
\begin{align*}
     \CU_n^2(W_{\a,-1,-1}) = \, & y   \CU_{n-1}^2(W_{\a,1,-1}) + (1-x-y) H_{n-1,2}(w_{1,\a}),\quad n\ge 1\\
     \CU_n^2(W_{-1,-1,-1}) = \, & x   \CU_{n-1}^2(W_{1,-1,-1}) + y (1-x-y) H_{n-2,1}(w_{1,1}), \quad n \ge 2.
\end{align*}
\end{cor}

\begin{rem}
In the statement of the theorem, we could also assume either $\a =-1$ or $\b = -1$ instead of $\g =-1$. 
The result then can be obtained via permutation of $(\a,\b,\g)$ and the corresponding permutation of
$(x,y,1-x-y)$. We give one example explicitly: 
\begin{align*}
   \CU_n^2(W_{-1,-1, \g}) =  x  y \CV_{n-2}^2(W_{1,1, \g}) + x H_{n-1,2}(w_{\g,1})  + y H_{n-1,1}(w_{\g,1}).
\end{align*}   
\end{rem}

Our second problem is to identify an inner product under which the eigenspaces in the 
Theorem \ref{thm:main_d=2_1} are orthogonal. 

\begin{defn} \label{thm:main_d=2_2} 
Let $\l_i$ and $\l_{i,j}$ be nonnegative numbers. Define the bilinear forms 
\begin{align*}
 \la f, g \ra_{\a,\b,-1}: = &\,  \int_{T^2} \left [x  \partial_x f(x,y) \partial_x g(x,y) + y \partial_y f(x,y) \partial_y g(x,y) \right]
               x^\a y^\b  dxdy \\
        &\, + \l_1 \int_0^1 f(x,1-x)g(x,1-x) x^\a(1-x)^\b dx \\
 \la f, g \ra_{\a,-1,-1}: = &\,   \int_{T^2} \partial_y f(x,y) \partial_y g(x,y)  x^\a dxdy, \\
        &\, + \l_1 \int_0^1 \partial_x f(x,0) \partial_x g(x,0) x^{\a+1} dx + \l_{1,0} f(1,0) g(1,0) \\
 \la f, g \ra_{-1,-1,-1}: = &\,   \int_{T^2} \partial_{xy} f(x,y) \partial_{xy} g(x,y) (1-x-y) dxdy, \\
        &\,  + \l_1 \int_0^1 \partial_x f(x,0) \partial_x g(x,0) dx + \l_2 \int_0^1 \partial_y f(0,y) \partial_y g(0,y) dy \\
        & \, + \l_{1,0} f(1,0) g(1,0) + \l_{0,1} f(0,1) g(0,1) + \l_{0,0} f(0,0) g(0,0).        
\end{align*}
\end{defn}

\begin{thm}
Let $\a,\b > -1$. For $n = 0,1,2, \ldots$, the following holds: 
\begin{enumerate}[$(1)$]
\item If $\l_1 > 0$ then  $\la f, g \ra_{\a,\b,-1}$ is an inner product for which $\CU_n^2(W_{\a,\b,-1})$ 
is the space
of orthogonal polynomials of degree $n$; 
\item If $\l_1 > 0$ and $\l_{1,0} > 0$, then  $\la f, g \ra_{\a,-1,-1}$ is an inner product for which 
$\CU_n^2(W_{\a,-1,-1})$ 
is the space of orthogonal polynomials of degree $n$; 
\item If $\l_1 > 0$, $\l_2 > 0$ and at least one of $\l_{1,0}, \l_{0,1}, \l_{0,0}$ is positive, then  
$\la f, g \ra_{-1,-1,-1}$ is an inner product for which $\CU_n^2(W_{-1,-1,-1})$ is the space of 
orthogonal polynomials of degree $n$, where we redefine $\CU_1^2(W_{-1,-1,-1})$ as
$$
  \CU_1^2(W_{-1,-1,-1}) = \mathrm{span} \left\{x - \frac{\l_{1,0}}{\l_{1,0}+\l_{0,1}+\l_{0,0}},
    y - \frac{ \l_{0,1}}{ \l_{1,0}+\l_{0,1}+\l_{0,0}}  \right\}.
$$  
\end{enumerate}
\end{thm}

With respect to $\la \cdot, \cdot \ra_{\a,\b,-1}$, the decomposition $\CU_n^2(W_{\a,\b,-1})$ is 
an orthogonal one, that is, 
$$
\CU_n^2(W_{\a,\b,-1}) =   (1-x-y) \CV_{n-1}^2(W_{\a,\b,1}) \bigoplus H_{n,3}(w_{\a,\b}), 
$$
as will be shown in (iii) of Theorem  \ref{thm:main_d>2_2}.

Again, by symmetry, we can state the inner product if other parameters are $-1$. We give one example. 
\begin{align*}
\la f, g \ra_{-1,-1, \g}: = &\, \int_{T^2} (\partial_y-\partial_x) f(x,y) 
    (\partial_y -\partial_x) g(x,y)  (1-x-y)^\g dxdy \\
	&\, + \l_1 \int_0^1 \partial_x f(x,0) \partial_x g(x,0) (1-x)^{\g+1} dx \\
	&\,+ \l_2 \int_0^1 \partial_y f(0,y) \partial_y g(0,y) (1-y)^{\g+1} dy + \l_{0,0} f(0,0) g(0,0) 
\end{align*}
where $\l_1, \l_2\ge 0$ with at least one of them positive and $\l_{0,0} > 0$.

\begin{rem} \label{rem:3.2}
Notice that $\la f, g \ra_{-1,-1, \g}$ contains two terms of one-dimenional integral, whereas 
$\la f, g \ra_{\a, -1,-1}$ contains only one such term. Evidently, by symmetry we could include
another term in $\la f, g \ra_{\a,-1,-1}$. In fact, it will become clear in the discussion below that
we could include even more terms if we are willing to use higher order derivatives. Apart from 
consideration of symmetry, as in the three vertices in $\la f, g \ra_{-1.-1.-1}$ in Definition
\ref{thm:main_d=2_2}, we shall keep the number of terms in the inner product minimal below. 
\end{rem}

These theorems will be special cases of the results on $T^d$. We will not give separate proof for $d =2$. 
Instead we give an indication below on how the main terms of the inner products are identitified. 
We start from an observation in \cite{BDFP} that the monic orthogonal polynomials 
$V_{k,n}^{(\a,\b,\g)}(x,y)$ are also orthogonal with respect to the inner product, 
when $\a, \b, \g > -1$, 
\begin{align} \label{eq:[f,g]}
  [ f, g ]_{\a,\b,\g} =  &\, \la f, g\ra_{\a,\b,\g} + \l_1 \la \partial_x f, \partial_x g \ra_{\a+1,\b,\g+1} 
       + \l_2 \la \partial_y f, \partial_y g \ra_{\a,\b+1,\g+1},  
\end{align}
which is a special case of Lemma \ref{lem:epd}. Since $V_{k,n}^{\a,\b,\g}$ are well-defined if
$\g = -1$, letting $\g \mapsto -1$ preserves the orthogonality, which is how the orthogonality 
in $\la f, g \ra_{\a,\b,-1}$ was established \cite{BDFP}, but the method does not give
the decomposition of the orthogonal space. Indeed, the second and the third terms in the right hand
side of $[ f, g ]_{\a,\b,\g}$ make sense for $\g =-1$ as well, whereas for the first term we 
parametrize the integral over $T^2$ as 
$$
    \int_{T^2} f(x,y) dxdy = \int_0^1 \int_0^{1-y} f(x,y) dx dy, 
$$
and apply the limit relation 
$$
       \lim_{\mu \to -1^+} \frac{1}{\mu+1} \int_0^1 f(x) x^{\mu} dx = f(0) 
$$
to conclude that 
\begin{align*}
   \lim_{\g \to -1^+} [f, g]_{\a,\b,\g} =&\, c_{\a,\b} \int_0^1 f(x,1-x) g(x,1-x) x^\a (1-x)^\b dx \\ 
       & +  \l_1 \la \partial_x f, \partial_x g \ra_{\a+1,\b,0} 
       + \l_2 \la \partial_y f, \partial_y g \ra_{\a,\b+1,0}, 
\end{align*}
which is a constant multiple of $\la f, g \ra_{\a,\b,-1}$ when $\l_2 = \l_1$. Since $V_{k,n}^{(\a,\b,-1)}(x)$ 
are well defined for all $0 \le k \le n$, their orthogonality is preserved under the above limit. 

The above limit process can be continued one more time by taking, for example, $\b \to -1$, which gives 
$$
   \lim_{\b, \g \to -1^+} [f, g]_{\a,\b,\g} = c \la f, g \ra_{\a,-1,-1}, 
$$
where $c$ is a constant. However, since $V_{k,n}^{(\a, -1, -1)}(x)$ are not well defined when $k= 0$, 
the limit process at this stage no longer gives the complete answer for our problem. 

If we take the limit $\a \to -1$ in $\la f, g \ra_{\a,-1,-1}$, we end up with 
\begin{align*}
   \lim_{\a \to -1^+} \la f, g \ra_{\a,-1,-1}  = & 
     \l_1 \int_0^1 \partial_x f(x,0) \partial_x g(x,0) dx +  \\ 
     &+  \l_2 \int_0^1 \partial_y f(0,y) \partial_y g(0,y) dy + \l_{1,0} f(1,0) g(1,0)
\end{align*}
which however is no longer an inner product since, for example, it is equal to zero if $f(x,y)= g(x,y) = x y$.
In order to obtain the main terms of the inner product in $\la f,g\ra_{-1,-1,-1}$ we need to go one level up 
by noticing that, by Lemma \ref{lem:epd},  $V_{k,n}^{(\a,\b,\g)}$ are orthogonal with respect to yet one 
more inner product: For $\a, \b, \g > -1$, 
\begin{align} \label{eq:[f,g]2}
  [ f, g ]_{\a,\b,\g} =  &\, \la f, g\ra_{\a,\b,\g} + \l_1 \la \partial_x f, \partial_x g \ra_{\a+1,\b,\g+1}    \\ 
    & + \l_2 \la \partial_y f, \partial_y g \ra_{\a,\b+1,\g+1} 
        + \l_{1,1}    \la \partial_{xy} f, \partial_{xy} g \ra_{\a+1,\b+1,\g+2} \notag
\end{align}
where $\l_1, \l_2, \l_{1,1}$ are nonnegative numbers. Taking the limit three times in the above 
inner product leads to the main terms of the inner product $\la f, g \ra_{-1,-1,-1}$. 

We emphasis, however, that the above limiting process only suggests possible form of the 
Sobolev inner product. It neither gives the final answer nor it proves the orthogonality. 

\subsection{Sobolev orthogonal polynomials on the simplex}

For a fixed positive integer $k$, $1 \le k \le d$, we consider $W_\bg$ with $\g_{d +2-k}= \ldots 
=\g_{d+1} = -1$. For convenience, we introduce the notation
$$
  W_{\g, -\one_k}(x) = W_\bg (x) \vert_{\g_{d+2 -k} =\ldots = \g_{d+1} = -1}, 
$$
where $\one_k := (1,\ldots,1) \in \ZZ^k$ and $-\one_k = (-1,\ldots,-1)$, and we always assume 
that the subindex of $W_\bg$ is a vector in $\RR^{d+1}$ so that, with the length of the vector 
$-\one_k$ being $k$, the length of $\g$ is $d-k+1$,  $\g = (\g_1,\ldots,\g_{d-k+1})$. In particular,  
$$
  W_{\g, -\one_k}(x) = x_1^{\g_1} \ldots x_{d-k+1}^{\g_{d-k+1}} x_{d-k+2}^{-1} \ldots x_d^{-1} (1-|x|)^{-1}.
$$
In these notations we can allow $k = d+1$ if we regard $\g \in \RR^{d-k+1}$ as nonentity 
when $k = d+1$ and write $W_{- \one_{d+1}}(x)  = W_{- \one}(x)$. 

Recall the subspace $H_{n,S_j}^d(W_\bg)$ in the Definition \ref{defn:H_n^d}.  From Remark \ref{rem}, 
since the values of $\g_i$ for $i \in S_{j}$ can be arbitraty in the basis of $H_{n,S_j}^d(W_\bg)$, we 
can set $ \{\g_i=0, \, i \in S_{j}\}$. Let
$$
S_{k,j}(\s):=\{\ell: \quad\varepsilon_{\ell}=0 \,\, \\,\,  
   (\varepsilon_{1},\ldots,\varepsilon_{d+1})=(\g,(\zero_{k-j},\one_j)\s)\},
$$
so that $S_{k,j}(\s)$ contains the indices of zero elements in $(\zero_{k-j},\one_j)\s$, and we
treat $S_{k,j}(\s)$ as a vector whose elements are arranged according to their order in 
$(\g,(\zero_{k-j},\one_j)\s)\}$. To simplify the notation, we define for $j =0,1,\ldots, k-1$,
$$
  H_{n}^d\left(W_{\g, (\zero_{k-j},\one_{j})\s}\right ): = 
       H_{n,S_{k,j}(\s)} ^d\left(W_{\g, (\zero_{k-j},\one_{j})\s} \right),
$$
where if $\s=id$, then $S_{k,0} = \{d+2-k, \ldots, d+1\}$ and $\one_0$ is a nonentity. It then follows 
by Lemma \ref{lem:H_n^d} that 
\begin{enumerate}[$(i)$]
\item $H_n^d(W_{\g, (\zero_{k-j},\one_{j})\s})$ is a subspace of $\CV_n^d (W_{\g,  (\zero_{k-j},\one_{j})\s} )$;
\item The restriction of $H_n^d(W_{\g, (\zero_{k-j},\one_{j})\s})$ on the face $T_{S_{k,j}(\s)}^{d-(k-j)}$ 
  agrees with the space $\CV_n^{d-(k-j)}({W_{\g, (\zero_{k-j},\one_{j})\s}}\big\vert_{T_{S_{k,j}(\s)}^{d-(k-j)}})$
  with variables in $\{x_i: i \in (S_{k,j}(\s))^c\}$.
\end{enumerate}
In particular, if $\s=id$, then $S_{k,j}=\{d-k+2,\ldots,d+1-j\}$ and
$$
H_n^d(W_{\g, (\zero_{k-j},\one_{j})})\big \vert_{{T^{d-(k-j)}_{S_{k,j}}}}
   =\CV_n^{d-(k-j)}\Big({W_{(\g, \zero_{k-j},\one_{j})}}\big\vert_{{T^{d-(k-j)}_{S_{k,j}}}}\Big)
$$
with variables $x_1,\ldots,x_{d-k+1}, x_{d+2-j},\ldots,x_d$.
Furthermore, by definition, we also have $H_{n}^d\left(W_{\g, \zero_k}\right ) = \emptyset$ 
if $k = d$ or $d+1$. 

We adopt the convention that 
$\sum_{i=j}^k  a_i = 0$ if $k < j$ regardless the values of $a_i$.

\begin{thm}\label{thm:main_d>2_1}
Let $k$ be a fixed integer, $1 \le k \le d+1$ and $\g \in \RR^{d-k+1}$. For $n \in \NN_0$, the differential 
equation 
\begin{align*}
           L_{\g,-\one_k} u =  - n(n+ |\g| + d - k) u 
\end{align*}
has a solution space  $\CU_n^d(W_{\g, -\one_k})$ of polynomials of degree $n$ with 
$$
   \dim  \CU_n^d(W_{\g, - \one_k})  = \binom{n+d-1}{n}, 
$$
which can be decomposed as a direct sum as follows: 
 
\noindent $(i)$ if $k =1$, then $\CU_0^d(W_{\g, -1}) = \mathrm{span} \{1\}$ and 
for $n \ge 1$, 
\begin{align} \label{V-decom-k=1}
    \CU_n^d (W_{\g, - 1}) =  x_{d+1}  \CV_{n-1}^d (W_{\g,1}) + H_n^d(W_{\g, 0}). 
\end{align}
 
\noindent $(ii)$ if $2 \le k \le d$, then $\CU_0^d(W_{\g, - \one_k}) = \mathrm{span} \{1\}$ and 
for $n \ge 1$, 
\begin{align} \label{V-decom}
    \CU_n^d (W_{\g, - \one_k}) = &\, x_{d+1-(k-1)} \cdots x_{d+1}  \CV_{n-k}^d (W_{\g, \one_k}) \\
    &   + \sum_{j=1}^{k-1} \sum_{\s \in \mathcal{G}_k}  x_{j,k}(\s) H_{n-j}^d\left(W_{\g, (\zero_{k-j},\one_{j})\s}\right)
   + H_n^d(W_{\g,\zero_k}), \notag
\end{align}
where $x_{d+1} = 1-|x|$, and with $x_0: =1$ and  $(\varepsilon_1,\ldots, \varepsilon_{d+1})
= (\zero_{d-k+1}, (\zero_{k-j},\one_{j})\s)$,
$$
    x_{j,k}(\s) :=\prod_{i =1}^{d+1} x_{\d_i} \quad\hbox{with}\quad 
       \d_i = i \,\, \hbox{if} \,\,  \ve_i =1, \,\, \hbox{and} \,\, \d_i = 0 \,\, \hbox{if} \,\, \ve_i =0.         
$$

\noindent $(iii)$ if $k = d+1$, then   $\CU_0^d(W_{-\one}) = \mathrm{span} \{1\}$,  
$\CU_1^d(W_{-\one}) = \mathrm{span} \{x_1+c_1,\ldots, x_d+c_d\}$ for arbitrary real numbers 
$c_1,\ldots, c_d$, and for $n \ge 2$,
\begin{align*}
 \CU_n^d(W_{- \one}) = x_1 \cdots x_{d+1}  \CV_{n-d-1}^d (W_{\one})  
         + \sum_{j=2}^{d} \sum_{\s \in \mathcal{G}_{d+1}}  x_{j,d+1}(\s) 
                  H_{n-j}^d\left(W_{(\zero_{d+1-j},\one_{j})\s}\right),
\end{align*}
where, with  $(\varepsilon_1,\ldots, \varepsilon_{d+1})= (\zero_{d+1-j},\one_{j})\s$,
$$
    x_{j,d+1}(\s) :=\prod_{i =1}^{d+1} x_{\d_i} \quad\hbox{with}\quad
        \d_i = i \,\, \hbox{if} \,\,  \ve_i =1, \,\, \hbox{and} \,\, \d_i = 0 \,\, \hbox{if} \,\, \ve_i =0.       
$$

\end{thm}

In the case of $k = d+1$, that is, all $\g_j = -1$, the differential equation becomes 
$$
           L_{-\one}  P = - n(n-1) P.
$$
For $n=1$, any polynomial of the form $P_i(x)=x_i+c_i$ with $c_i$ being an arbitrary constant 
satisfies the equation $L_{-\one} P =  0$. The values of the constant will be fixed when we
discuss the Sobolev orthogonality.

As an example, we give the decomposition for $k =3$ and $k < d$ explicitly: 
\begin{align*}
   \CU_n^d(W_{\g, -1,-1,-1}) =  & \, x_{d-1} x_d x_{d+1}  \CV_{n-3}^d (W_{\g,1,1,1}) 
       + x_{d-1} x_d H_{n-2}^d\left(W_{\g,1,1,0}\right) \\
   &  + x_{d-1} x_{d+1} H_{n-2}^d\left(W_{\g,1,0,1}\right) +  x_{d} x_{d+1} H_{n-2}^d\left(W_{\g,0,1,1}\right) \\
   &   + x_{d-1} H_{n-1}^d\left(W_{\g,1,0,0}\right) + x_d H_{n-1}^d\left(W_{\g,0,1,0}\right) \\
   &   + x_{d+1} H_{n-1}^d\left(W_{\g,0,0,1}\right) + H_{n}^d(W_{\g,0,0,0}).
\end{align*}

Our next theorem shows that the elements in $\CU_n^d(W_{\g, -\one_k})$ are orthogonal polynomials
with respect to an inner product of the Sobolev type. We need the following notations: for 
$1 \le k \le d-1$, define 
$$
\xb_k:=(x_1,\ldots,x_k),\quad |\xb_k|=x_1+\ldots+x_k,\quad \bg_k:=(\g_1,\ldots,\g_k).
$$
Recall that, for $I \subset \ZZ_{d+1}$ with $|I| =i \le d$,  the face $T_I^{d-i}$ of  the simplex $T^d$ 
is defined by $T_I^{d-i} = \{x \in T^d: x_\ell =0, \ell \in I\}$. If $I = \{j_1,\ldots,j_i\}$, we define
$\xb_I := (x_{j_1},\ldots, x_{j_i})$.  

\begin{defn} \label{defn:ipd}
For $\g \in \RR^d$, $\g_{d+1} = -1$ and $\l \ge 0$, define 
\begin{align*}
  \la f, g\ra_{\g, -1}:=&\, \sum_{i=1}^{d} \int_{T^d} x_i \frac{\partial f}{\partial x_i} \frac {\partial g}{\partial x_i} x^{\g} dx 
 +\l \int_{T^{d-1}_{\{d+1\}}} f(x) g(x) \xb_{d-1}^{\bg_{d-1}} (1-|\xb_{d-1}|)^{\g_d} d \xb_{d-1}.
%
\end{align*}
For $2 \le k \le d$, let $\mb_k=\left\{d-k+2,\ldots, d\right\}$ and $\mb_{k}^{+}=\left\{d-k+2,\ldots, d+1\right\}$. For $\g \in \RR^{d+1-k}$ and $\l, \l_i, \l_I \ge 0$, 
define
\begin{align*}
  \la f, g\ra_{\g, -\one_k} := & \int_{T^d} \frac{\partial^{k-1}f}{\partial x^{\mb_k}} \frac{\partial^{k-1}g}
      {\partial x^{\mb_k}} (1- |x|)^{|\mb_k| -1} \xb_{d+1-k}^\g dx  \\ 
& +\sum_{i=1}^{k-2} \sum_{\substack{ I\subset {\mb}_{k}  \\ |I| = i}} {\l_I} \int_{T_{I^c}^{d-k+i+1}} 
    \frac{\partial^{i}f}{\partial x^{I}}\frac{\partial^{i}g}{\partial x^{I}}  (1- |x|)^{i-1} \xb_{d+1-k}^\g 
    d\xb_{d-k+1} d \xb_I \\
& + \sum_{i=1}^{d-k+1} \l_i \int_{T_{\mb_k}^{d-k+1}} x_i 
        \frac{\partial f} {\partial x_i} \frac{\partial g}{\partial x_i}\xb_{d+1-k}^\g d\xb_{d-k+1}\\
& +\l \int_{T^{d-k}_{\mb_{k}^{+}}} f(\xb_{d-k+1},\zero) \ g(\xb_{d-k+1},\zero) \xb_{d-k+1}^{\g}
   d\xb_{d-k},
\end{align*}
where if $k=d$, we replace the last term by $\l f(e_1) g(e_1)$, with $\l \ge 0$.
Finally, when all $\g_i=-1$, for $\l_I \ge 0$ and  $\l_{i,0} \ge 0$, define
\begin{align*}
  \la f, g\ra_{ -\one} := & \int_{T^d} \frac{\partial^{d}f}{\partial x^{\ZZ_{d}}} \frac{\partial^{d}g}
      {\partial x^{\ZZ_{d}}} (1- |x|)^{d-1} dx  \\ 
&+\sum_{i=1}^{d-1} \sum_{\substack{ I\subset {\ZZ}_{d}  \\ |I| = i}} {\l_I} \int_{T_{I^c}^{i}} 
      \frac{\partial^{i}f}{\partial x^{I}}\frac{\partial^{i}g}{\partial x^{I}}  (1- |x|)^{i-1} d\xb_{I} \\
& + \l_{0,0} f(e_0) g(e_0)+ \l_{1,0} f(e_1) g(e_1)+\ldots +  \l_{d,0} f(e_d) g(e_d). 
\end{align*}
\end{defn}

In the following we shall write $\la f, g \ra_{\g, - \one_k}$ with $k = 1,2,\ldots, d+1$ for all 
three cases. Whenever  $\la f, g \ra_{\g, - \one_k}$ is an inner product, we can define
a space of orthogonal polynomials. Our next result says that the space of orthogonal 
polynomials with respect to $\la f, g \ra_{\g, - \one_k}$ is exactly $\CU_n^d(W_{\g, - \one_k})$.
First, however, we need to specify $\CU_1^d(W_{-\one})$: We redefine
$$
  \CU_1^d(W_{-\one}) = \mathrm{span} \{x_1+c_1,\ldots, x_d+ c_d\} \quad \hbox{with} \quad 
      c_j = -\frac{\l_{j,0}}{ \l_{0,0}+\l_{1,0}+\ldots+\l_{d,0}}. 
$$

\begin{thm}\label{thm:main_d>2_2}
For $1 \le k \le d+1$, let $\g \in \RR^{d+1-k}$ with $\g_j > -1$, $1 \le j \le d+1-k$. 
\begin{enumerate}[$(i)$]
\item $\la f, g \ra_{\g,-\one_k}$ is an inner product if $\l, \l_i, \l_I>0$ for $1 \le k \le d$ and,
in addition, at least one of $\l_{j,0}$ is positive for $k = d+1$. 
\item $\CU_n^d(W_{\g,-\one_k})$ is the space of orthogonal polynomials with respect to 
$\la f, g \ra_{\g,-\one_k}$.
\item In the case of $k =1$, the decomposition of \eqref{V-decom-k=1} is orthogonal, 
that is, 
$$
    \CU_n^d (W_{\g, - 1}) =  x_{d+1}  \CV_{n-1}^d (W_{\g,1})  \bigoplus H_n^d(W_{\g, 0}). 
$$
\end{enumerate}
\end{thm}

\begin{rem}
For $ k \ge 2$, the direct sums in the decomposition of $\CU_n^d(W_{\g,-\one_k})$ is in 
general not orthogonal sums. In other word, the elements belong to different parts of the 
decomposition \eqref{V-decom} may not be mutually orthogonal in general. On the 
other hand, the Gram-Schmid process is applicable if an orthonormal basis is desired.
\end{rem}

It should be mentioned that, as stated in Remark \ref{rem:3.2}, it is possible to add more
terms or higher order derivatives in the definition of $\la f, g \ra_{\g, -\one_k}$ but we 
strike to keep the number minimal.

\section{Lemmas and Proof of Theorems}
\setcounter{equation}{0}

In the first subsection we prove several lemmas that will be needed for the proof of Theorem \ref{thm:main_d>2_1}. 
The proofs of the main results are given in the subsequent sections. 

\subsection{Lemmas on Rodrigue basis}

Our first two lemmas are analogues of \eqref{Jacobi-negative} of one variable with one of the element 
in $\bg$ being $-1$.

\begin{lem}\label{lem:d1}
Let $\g \in \RR^d$ and $\g_{d+1}=-1$. Then, for $n_i \ge 1$, $1 \le i \le d$, 
\begin{align*}
P_{\nb} ^ {\left (\g,-1\right)} (x)=(1-|x|) \sum_{i=1}^{d} \frac{n_i (\g_i+n_i)} {|\nb|} P_{\nb-e_i}^ {\left (\g,1\right)}(x).
\end{align*}
\end{lem}

\begin{proof}
Applying the product rule
\begin{align*}
\left(\frac{d} {dx} \right)^n\left[(1-x)g(x)\right]=(1-x)\frac{d^n}{dx^n}g(x)-n\frac{d^{n-1}}{dx^{n-1}}g(x).
\end{align*}
multiple times within the Rodrigue formula, we obtain 
\begin{align*}
P_{\nb} ^ {\left (\g,-1\right)} (x)=  x^{-\g} \frac{\partial^{|\nb|}} {\partial x^{\nb}} \left[x^{\g + \nb} (1-|x|)^{|\nb|} \right]   
 +x^{-\g} \sum_{i=1}^{d} n_i \frac{\partial^{|\nb|-1}} {\partial x^{\nb-e_i}}\left[x^{\g + \nb} (1-|x|)^{|\nb|-1} \right],
\end{align*}
in which the first term in the right hand can be written as
\begin{align*}
x^{-\g} \frac{\partial^{|\nb|}} {\partial x^{\nb}} \left[x^{\g + \nb} (1-|x|)^{|\nb|} \right]
 &\, = \frac{(n_1+\cdots +n_d)} {|\nb|} x^{-\g}    \frac{\partial^{|\nb|}} {\partial x^{\nb}} \left[x^{\g + \nb} (1-|x|)^{|\nb|} \right]\\
&\, =\frac {x^{-\g}} {|\nb|} \sum_{i=1}^{d} n_i \frac{\partial^{|\nb|-1}} {\partial x^{\nb-e_i}} 
   \left[ \frac {\partial} {\partial x_i} \left(x^{\g + \nb} (1-|x|)^{|\nb|} \right)\right],
\end{align*}
so that, putting the two identities together,  we conclude that 
\begin{align*}
   P_{\nb} ^ {\left (\g,-1\right)} (x) &\,=  \frac {x^{-\g}} {|\nb|} \sum_{i=1}^{d} {n_i (\g_i+n_i)} 
            \frac{\partial^{|\nb|-1}} {\partial x^{\nb-e_i}}\left[x^{\g + \nb-e_i} (1-|x|)^{|\nb|} \right]\\
&\, =(1-|x|) \sum_{i=1}^{d} \frac{n_i (\g_i+n_i)} {|\nb|} P_{\nb-e_i}^ {\left (\g,1\right)}(x),
\end{align*}
which completes the proof.
\end{proof}

\begin{lem}\label{lem:d2}
Assume $\g_i=-1$ for some $i$, $1\le i\le d$. Then for $n_i \ge 1$, 
\begin{align*}
P_{\nb} ^ {(\g,\g_{d+1})} (x)=-(|\nb|+\g_{d+1}) x_i P_{\nb-e_i}^ {(\g+2e_i,\g_{d+1})}(x). 
\end{align*}
\end{lem}

\begin{proof}
Without loss of generality, we can assume $i =d$. Setting $\g:=(\bg_{d-1},-1)$ and
$x=(\xb_{d-1}, x_d)$, and applying the product rule
\begin{align*}
\left(\frac{d} {dx} \right)^n\left[xg(x)\right]=x\frac{d^n}{dx^n}g(x)+n\frac{d^{n-1}}{dx^{n-1}}g(x),
\end{align*}
multiple time with the Rodrigue formula, we obtain 
\begin{align*}
P_{\nb} ^ {\left (\g,\g_{d+1}\right)} (x) = &\, \xb_{d-1}^{-\bg_{d-1}} (1-|x|)^{-\g_{d+1}}\frac{\partial^{|\nb|}} {\partial x^{\nb}} \left[x^{\g+ \nb+e_d} (1-|x|)^{|\nb|+\g_{d+1}} \right]\\
&\, -n_d \xb_{d-1}^{-\bg_{d-1}} (1-|x|)^{-\g_{d+1}}\frac{\partial^{|\nb|-1}} {\partial x^{\nb-e_d}} \left[x^{\g + \nb} (1-|x|)^{|\nb|+\g_{d+1}} \right].
\end{align*}
Splitting the partial derivatives in the first term as  $\frac{\partial^{|\nb|}} {\partial x^{\nb}}
  = \frac{\partial^{|\nb|-1}} {\partial x^{\nb-e_d}} \frac{\partial} {\partial x_d}$, a quick computations
establishes the stated identity. 
\end{proof}

By repeatedly applying the above lemmas, we can deduce a relation when more than one $\g_i$ equal
to $-1$. We state the results in the following corollaries. 

\begin{cor}\label{cor:d0}
Let  $\g\in\RR^{d+1-k}$. For  $ n_{d-k+2} \ge 1, \ldots, n_d \ge 1$, 
\begin{align*}
P_{\nb}^{(\g,-\one_{k-1},\g_{d+1})} (x)= & (-1)^{k-1}  x_{d+2 -k} \cdots x_{d} \\
     & \times  \prod_{j=1}^{k-1}(|\nb|+\g_{d+1}-j+1) P_{\nb-(\zero, \one_{k-1})}^{(\g,\one_{k-1},\g_{d+1})} (x).
\end{align*}
In particular, it follows that 
$$
x_{d+2 -k} \cdots x_{d}\CV_{n-k+1}^d(W_{\g, \one_{k-1},\g_{d+1}})\subset \CV_{n}^d(W_{\g, -\one_{k-1},\g_{d+1}}).
$$
\end{cor}
 
Combining Lemmas \ref{lem:d1} and \ref{lem:d2} , we can also deduce the following lemma: 
 
\begin{cor}\label{cor:d1}
Let $\g \in \RR^{d+1-k}$. For $n_{d-k+2} \ge 1,\ldots,n_d \ge 1$, 
\begin{align}\label{eq2}
P_{\nb}^{(\g,-\one_k)}(x)=x_{d+2 -k} \cdots x_{d} (1-|x|)\sum_{i=1}^{d} \l_{n_i, \g_i}
     P_{\nb-e_i-(\zero, \one_{k-1})}^{(\g,\one_k)} (x),
\end{align}
where, for $1 \le i \le d$,
$$
   \l_{n_i,\g_i}=(-1)^{k-1} \frac{n_i(n_i+\g_i)}{|\nb|}\prod_{j=0}^{k-2} (|\nb|-j). 
$$
\end{cor}

For later use, we need to reverse the expression in \eqref{eq2}, which is the content of the
following lemma. 

\begin{lem}\label{cor:d2}
Let $\jb=(j_2,\ldots,j_d)$. Then there exist constant $\mu_\jb$ such that 
\begin{align}\label{eq3}
  x_{d+2 -k} \cdots x_{d} (1-|x|)P_{\nb} ^ {(\g,\one_k)}(x)=\sum_{j_i\le n_i}\mu_\jb 
        P_{(|\nb|-|\jb|+1, \jb)+(\zero,\one_{k-1})}^{(\g,-\one_k)} (x), 
\end{align}
where the sum is over $j_i$ for $i=2,3,\ldots,d$. In particular, for $n \ge k$, 
$$
x_{d+2 -k} \cdots x_{d} (1-|x|)\CV_{n-k}^d(W_{\g, \one_k})\subset \CV_{n}^d(W_{\g, -\one_k}).
$$
\end{lem}

\begin{proof}
To simplify the notation in the proof, let 
$$
X_k :=x_{d-k+2}\cdots x_d (1-|x|), \quad \nb_k:=(n_1,\ldots, n_k) \in \NN_0^k. 
$$
From \eqref{eq2} it follows that 
$$
  P_{(n,\zero_{d-k},\one)}^{(\g,-\one_k)}=X_k \l_{n,\g_1} P_{(n-1,\zero)}^{(\g,\one_k)} (x),
$$
which proves \eqref{eq3} for $\nb=(n,0,\ldots,0)$. Notice that the subindex $\nb$ of $P_\nb^\bg$
is always in $\NN_0^d$, so that we have suppressed the subindex of $\one = \one_{k-1}$ for the polynomial
in the left hand side and the subindex of $\zero = \zero_{d-1}$ for the polynomial in the right hand side. We 
shall keep this convention below. Again by \eqref{eq2}, we have 
$$
P_{(n_1,n_2,\zero_{d-(k+1)},\one)}^{(\g,-\one_k)}=X_k\left[\l_{n_1,\g_1} P_{(n_1-1,n_2,\zero)}^{(\g,\one_k)} (x)
        +\l_{n_2,\g_2} P_{(n_1,n_2-1,\zero)}^{(\g,\one_k)} (x)\right],
$$
which shows by induction on $n_2$ that
$$
  X_k P_{(n_1,n_2,\zero)}^{(\g,\one_k)} (x)=\sum_{j=0}^{n_2}\mu_j 
         P_{(n_1+n_2-j+1,j,\zero_{d-(k+1)},\one)}^{(\g,-\one_k)}(x),
$$
where the $\mu_j$ can be determined inductively, which proves \eqref{eq3} for $\nb=(n_1,n_2,\ldots,0)$. 
Furthermore, the above proof shows, in fact, that
$$
  X_k P_{({(n_1,n_2,\zero_{d-(k+1)})}\s,\zero)}^{(\g,\one_k)} (x)=\sum_{j=0}^{n_2}\mu_j 
    P_{({(n_1+n_2-j+1,j,\zero_{d-(k+1)})}\s,\one)}^{(\g,-\one_k)}(x)
$$
for every $\s \in \mathcal{G}_{d-k+1}$, the permutation group of $(d-k+1)$ elements. In the next step, we 
deduce from \eqref{eq2} that 
\begin{align*}
P_{(n_1,n_2,n_3,\zero_{d-(k+2)},\one)}^{(\g,-\one_k)}(x) 
    = &\, X_k \ [ \l_{n_1,\g_1} P_{(n_1-1,n_2,n_3,\zero)}^{(\g,\one_k)} (x)\\
     &   +\l_{n_2,\g_2} P_{(n_1,n_2-1,n_3,\zero)}^{(\g,\one_k)} (x)
              + \l_{n_3,\g_3} P_{(n_1,n_2,n_3-1,\zero)}^{(\g,\one_k)}(x) ].
\end{align*}
For $n_3=1$, we can use induction on $n_2$ to deduce, from 
$X_k P_{({(n_1,n_2,\zero_{d-(k+1)})}\s,\zero)}^{(\g,\one_k)} (x)$, 
the desired \eqref{eq3} for $X_k P_{(n_1,n_2,1,\zero)}^{(\g,\one_k)} (x)$ and, hence, for 
$X_k P_{({(n_1,n_2,1,\zero_{d-(k+2)})\s},\zero)}^{(\g,\one_k)} (x)$. Then, by induction on $n_3$,
we can deduce the desired \eqref{eq3} for $\nb =(n_1,n_2,n_3,\zero)$. It is now clear how to 
proceed, by induction if necessary, to conclude that \eqref{eq3} holds for $\nb = (\nb_{d-k},\zero)$. 

We can go beyond $d-k$ by using \eqref{eq2} and the fact that $\l_{m,\tau} = 0$ if $m =1$ and 
$\tau =-1$. Indeed,  \eqref{eq2} implies that
$$
   P_{(\nb_{d-k+1},\one)}^{(\bg, - \one_k)}(x) = X_k  \sum_{i=1}^{d-k+1} \l_{n_i,\g_i}
       P_{(\nb_{d-k+1},\zero) -e_i}^{(\bg,  \one_k)}(x),    
$$
which allows us to use the induction on $n_{d-k+1}$ to prove \eqref{eq3} for 
$\nb = (\nb_{d-k+1},\zero)$. Furthermore, the process can obviously be continued, by
induction if necessary, until we reach $\nb =(n_1,\ldots, n_d)$. This completes the proof. 
\end{proof}
 
\subsection{Proof of Theorem \ref{thm:main_d>2_1}}
By  \eqref{V-decom} in the Theorem \ref{thm:main_d>2_1}, the space $\CU_n^d(W_{\g, -\one_k})$ 
consists of three main terms. Since, by Lemma \ref{cor:d2}, the elements in the space 
$x_{d+2-k} \cdots x_{d+1}  \CV_{n-k}^d (W_{\g, \one_k})$ are orthogonal polynomials of degree
$n$ in $\CV_n^d(W_{\g, -\one_k})$, by Remark \ref{remark:g<-1}, they satisfy the equation
\begin{align}\label{diff}
           L_{\g,-\one_{k}} u =  - n(n+ |\g| + d - k) u,
\end{align}
which takes care of the first term. The second term of the decomposition in \eqref{V-decom} is a 
large sum. Since the permutation is acting on $(\zero_{k-j},\one_j)$ and the differential operator
$L_{\g,\zero_{k-j},\one_j} $ is clearly invariant under such permutation, we only need to consider,
for $1 \le j \le k-1$, the cases of $x_{j,k} H_{n-j}^d(W_{\g,\zero_{k-j},\one_j})$ with $x_{j,k}
=x_{d+2-j}\cdots x_{d+1}$ and  $x_{j,k} H_{n-j}^d(W_{\g,\zero_{k-j-1},\one_j,0})$ with 
$x_{j,k}=x_{d+1-j}\cdots x_{d}$. We consider the case of 
$x_{d+2-j}\cdots x_{d+1}H_{n-j}^d(W_{\g,\zero_{k-j},\one_j})$ first. Recall that 
$$
H_{n-j}^d(W_{\g, \zero_{k-j},\one_j})=H_{{n-j},S_{k,j}}^d(W_{\g, \zero_{k-j},\one_j})
\quad\textit{with}\quad S_{k,j}=\{d-k+2,\dots,d-j+1\}.
$$
From Lemma \ref{rem}, since the parameters in $ \{\g_i: i\in S_{k,j}\}$ do not appear 
in the basis of $H_{n-j}^d(W_{\g, \zero_{k-j},\one_j})$, the values of these parameters 
can be arbitrary. As a result, we can write 
\begin{align}\label{equa}
H_{n-j}^d(W_{\g, \zero_{k-j},\one_j})=H_{n-j}^d(W_{\g,-\one_{k-j},\one_j}).
\end{align}
On the other hand, from Lemma \ref{cor:d2}, for $\g \in \RR^{d-k+1}$, we have
$$
x_{d+2 -j} \cdots x_{d} (1-|x|)\CV_{n-j}^d(W_{\g, -\one_{k-j},\one_j})\subset \CV_{n}^d(W_{\g, -\one_k}).
$$
Since, by Lemma \ref{lem:H_n^d}, $H_{n-j}^d(W_{\g, -\one_{k-j},\one_j})\subset 
\CV_{n-j}^d(W_{\g, -\one_{k-j},\one_j})$, together with \eqref{equa} it follows that 
$$
x_{d+2 -j} \cdots x_{d} (1-|x|)H_{n-j}^d(W_{\g, \zero_{k-j},\one_j})\subset \CV_{n}^d(W_{\g, -\one_k}),
$$
so that the elements of $x_{d+2 -j} \cdots x_{d} (1-|x|)H_{n-j}^d(W_{\g, \zero_{k-j},\one_j})$ 
satisfy the equation \eqref{diff}. Secondly, we consider the case of 
$x_{d+1-j}\cdots x_{d}H_{n-j}^d(W_{\g,\zero_{k-j-1},\one_j,0})$. Similar to the first case,  as a 
result of Corollary \ref{cor:d0} and Lemma \ref{rem}, we can write
$$
x_{d+1 -j} \cdots x_{d}H_{n-j}^d(W_{\g,\zero_{k-j-1},\one_j,0})\subset \CV_{n}^d(W_{\g, -\one_k}),
$$
so that these polynomials satisfy the equation \eqref{diff} for $ j=1,2,\ldots,k-1$. 
Finally,  the elements of the third term $H_{n}^d(W_{\g,\zero_k})$ in the decomposition of 
$\CU_n^d(W_\g,  - \one_{k})$ satisfy the equation \eqref{diff} according to Lemma \ref{rem}.  

The same proof applies to the case of $k=d+1$ for $n > 1$, whereas the case 
$\CU_1^d(W_{-\one})$ has been explained right below the statement of  
Theorem \ref{thm:main_d>2_1}. 

It is remain to prove that the $\CU_n^d(W_{\g, -\one_k})$ has full dimension. From 
Lemma \ref{lem:H_n^d}, the restriction of $H_n^d(W_{\g, (\zero_{k-j},\one_j)\s})$ 
on the face $T_{S_{k,j}(\s)}^{d-(k-j)}$ is isomorphic to $\CV_n^{d-(k-j)}(W_{\g, \one_{j}})$, so that 
 \begin{align*}
   \dim  H_{n}^d(W_{\g, (\zero_{k-j},\one_j)\s}) = \binom{n+d-(k-j+1)}{n},
\end{align*}
which implies, together with $\dim  \CV_n^d(W_{\bg})  = \binom{n+d-1}{n}$, that 
\begin{align*}
\dim \CU_n^d(W_{\g,-\one_{k}})=&\, \dim \CV_{n-k}^d (W_{\g, \one_k}) \\
   &\, + \sum_{j=1}^{k-1}\binom{k}{j}\dim  H_{n-j}^d(W_{\g, (\zero_{k-j},\one_j)\s})
      +\dim  H_{n }^d(W_{\g, \zero_{k}}) \\
  = & \,\sum_{j=0}^{k}\binom{k}{j}\binom{n+d-k-1}{n-j}=\binom{n+d-1}{n},
\end{align*}
where the last step uses the well-known combinatorial identity.  \qed

\subsection{Proof of (i) in Theorem \ref{thm:main_d>2_2}}

To show that each bilinear form is an inner product, we only need to show that it preserves 
positivity. Since all coefficients are assumed to be non-negative, we clearly have 
$\la f, f \ra_{\g,-\one_k}\ge 0$ and it remains to show that $\la f, f\ra_{\g,-\one_k}=0$ implies
that $f \equiv 0$. 

Assume that $k=1$ and $\la f, f \ra_{\g,-1}= 0$. Then, by the definition, $\frac{\partial f}{\partial x_i}=0$
for $1 \le i \le d$, which implies that $f$ is a constant. Furthermore,  the last term of $\la f, f\ra_{\g, -1}$
shows then $f \equiv 0$. Thus, $\la f, f \ra_{\g,-1}$ is an inner product. 

Now assume that $2\le k\le d$ and $\la f, f \ra_{\g,-\one_k}=0$. Then each term of in the right hand side 
of  $\la f, f \ra_{\g,-\one_k}$ is zero. There are four main terms. The first term shows 
$\frac{\partial^{k-1}f}{\partial x^{\mb_k}}=0$, which implies that $f$ has the form
$$
   f(x)=\sum_{\substack{ J\subset {\mb}_{k}  \\ |J| =k-2 }} f_{J} (x_J),\qquad x_J\in {T_{J^c}^{d-1}}. 
$$
The second term with $|J| = k -2$ shows $\frac{\partial^{k-2}f}{\partial x^{J}}=0$, which implies that
$f$ has the form
\begin{align*}
f(x)=\sum_{\substack{ J\subset {\mb}_{k}  \\ |J| =k-3 }} f_{J} (x_J),\quad x_J\in {T_{J^c}^{d-2}},
\end{align*}
we can continue this process with $|J| = j$ for $j = k -3, \ldots, 1$ and deduce from the second 
term of $\la f,f\ra_{\g, -\one_k}$ that $f$ has the form  $f(x) = f_1(x_1,\dots,x_{d-k+1})$. Now, 
the third term of $\la f,f\ra_{\g, -\one_k}$ shows that $\frac{\partial f} {\partial x_j}=0$ for $j=1,\ldots,d-k+1$, 
from which it follows then that $f$ is a constant. Finally, the fourth term of $\la f,f\ra_{\g, -\one_k}$ shows 
that $f(x)=0$ if $k < d$ or if $k = d$ and $\l>0$. Thus $\la f, g \ra_{\g,-\one_k}$ is an inner product. The same proof clearly applies to $\la f, g\ra_{-\one}$. \qed

\subsection{Proof of (ii) in Theorem \ref{thm:main_d>2_2}}
We need to show that, for each $k$, the polynomial spaces of $ \CU_n^d(W_{\g,-\one_k})$ 
given in Theorem \ref{thm:main_d>2_1} are the spaces of orthogonal polynomials with respect 
to the inner product  $\la \cdot, \cdot\ra_{\g, -\one_k}$. The proof is involved and will be divided
into several cases. Throughout the proof, we let $g$ be a generic element in $\Pi_{n-1}^d$.

\medskip \noindent
{\bf Case 1. $k=1$}. The inner product $\la f, g\ra_{\g,-1}$ contains two terms, 
$$
  \la f, g\ra_{\g,-1} =  [ f, g]_1 + [f,g]_2,
$$
where 
\begin{align*}
 [f, g ]_1 :=  &\, \sum_{i=1}^{d} \int_{T^d} x_i \frac{\partial f}{\partial x_i} \frac {\partial g}{\partial x_i} x^{\g} dx \\ 
  [f,g]_2 := &\,\l \int_{T^{d-1}_{\{d+1\}}} f(x) g(x) \xb_{d-1}^{\bg_{d-1}} (1-|\xb_{d-1}|)^{\g_d} d \xb_{d-1}.
\end{align*}
By Theorem \ref{thm:main_d>2_1}, the space of $ \CU_n^d(W_{\g,-1})$ 
with $\g_i > -1$, $1 \le i \le d$, is decomposed into two terms as shown in \eqref{V-decom}. 
The proof is divided into two cases accordingly. 
 
\medskip \noindent
{\it Case 1.1}. Let $f(x)=(1-|x|)P_{\nb}^{(\g,1)}(x)$, where $P_{\nb}^{(\g,1)}(x)\in\CV_{n-1}^d(W_{\g,1})$
and $\nb \in \NN_0^d$ and $|\nb|=n-1$. Then the restriction of $f$ on $|x|=1$ is zero so that 
$[f,g]_2= 0$ for all $g$. Furthermore, an integration by parts gives
\begin{align} \label{[f,g]-case1.1}
  [ f, g]_1 = - \sum_{i=1}^{d} \int_{T^d} x^{\g}(1-|x|) P_{\nb}^{(\g,1)}(x)\left\{x_i \frac{\partial ^2g}{\partial {x_i}^2}+(\g_i +1)\frac {\partial g}{\partial x_i}\right\} dx.
\end{align}
Since the term in the curly bracket is a polynomial of degree at most $n-2$, the orthogonality
of  $P_{\nb}^{(\g,1)}(x)\in\CV_{n-1}^d(W_{\g,1})$ implies immediately that $[f,g]_1 =0$, so
that $\la f,g\ra_{\g,1} = 0$ in this case. 

\medskip \noindent
{\it Case 1.2}. Let $f \in H_n^d(W_{\g,0})$. By Lemma \ref{lem:H_n^d}, the restriction of 
$H_n^d(W_{\g,0})$ on the plane $|x|=1$ is exactly $\CV_{n}^{d-1}(W_{\g})$. It follows 
immediately that $[f,g]_2 = 0$ for all $g \in \Pi_{n-1}^d$. Furthermore, an integration by 
parts gives
\begin{align*}
  [f, g]_1 =-\sum_{i=1}^{d} \int_{T^d} x^{\g}f(x)\left\{x_i \frac{\partial ^2g}{\partial {x_i}^2}
      +(\g_i +1)\frac {\partial g}{\partial x_i}\right\} dx.
\end{align*}
Since $f \in \CV_{n}^{d}(W_{\g,0})$ by Lemma \ref{lem:H_n^d}, it follows that $[f,g]_1 =0$,
so that $\la f, g\ra_{\g, -1}=0$. 

By the decomposition of $\CU_n^d(W_{\g,-1})$, the two cases imply that 
$\la f, g\ra_{\g, -1}=0$ for all $f \in \CU_n^d(W_{\g,-1})$ and $g\in \Pi_{n-1}^d$, which
completes the proof in this case. 

\medskip\noindent 
{\bf Case 2}. $2\le k \le d$. The inner product $\la f, g\ra_{\g,-\one_k}$ contains four terms, 
\begin{equation} \label{ipd-decomp}
  \la f, g\ra_{\g,-1} =  [ f, g]_1 + [f,g]_2+ [ f, g]_3 + [f,g]_4
\end{equation}
where 
\begin{align*}
[ f, g]_1 := \,& \int_{T^d} \frac{\partial^{k-1}f}{\partial x^{\mb_k}} \frac{\partial^{k-1}g}
      {\partial x^{\mb_k}} (1- |x|)^{|\mb_k| -1} \xb_{d+1-k}^\g dx, \\ 
[ f, g]_2 := \,&\sum_{i=1}^{k-2} \sum_{\substack{ I\subset {\mb}_{k}  \\ |I| = i}} {\l_I} \int_{T_{I^c}^{d-k+i+1}} 
    \frac{\partial^{i}f}{\partial x^{I}}\frac{\partial^{i}g}{\partial x^{I}}  (1- |x|)^{i-1} \xb_{d+1-k}^\g d\xb_{d-k+1}
      d\xb_I,\\
[ f, g]_3 := \,& \sum_{i=1}^{d-k+1} \l_i \int_{T_{\mb_k}^{d-k+1}} x_i 
        \frac{\partial f} {\partial x_i} \frac{\partial g}{\partial x_i}\xb_{d+1-k}^\g d\xb_{d-k+1},\\
[ f, g]_4 := \,&\l \int_{T^{d-k}_{\mb_{k}^{+}}} f(\xb_{d-k+1},\zero) \ g(\xb_{d-k+1},\zero) \xb_{d-k+1}^{\g}d\xb_{d-k},
\end{align*} 
where $\mb_k=\left\{d-k+2,\ldots, d\right\}$ and $\mb_{k}^{+}=\left\{d-k+2,\ldots, d+1\right\}$.
By Theorem  \ref{thm:main_d>2_1}, the space $\CU_n^d(W_{\g,-\one_k})$, $\g_i > -1$ for 
$1 \le i \le d+1-k$, is decomposed into three pieces as in \eqref{V-decom}, where the second term
is a large sum. The proof is divided into several cases according to this decomposition. 

\medskip\noindent
{\it Case 2.1}. $f$ is an element of $x_{d-k+2} \cdots x_{d+1} \CV_{n-k}^d(W_{\g, \one_k})$, the first
term of \eqref{V-decom}. Let $f(x ) = x_{d-k+2}\cdots x_d x_{d+1}P_{\nb}^{(\g,\one_k)}(x)$, where 
$x_{d+1}=(1-|x|)$ and $P_{\nb}^{(\g,\one_k)}(x)\in\CV_{n-k}^d(W_{\g,\one_k})$ with 
$\nb \in \NN_0^d$ and $|\nb|=n-k$. 

For $I \subset \mb_k$ with $|I| = i$, the variables on indices in $I^c = \mb_k \setminus I$ appear 
as product factors in $\frac{\partial^{i}f}{\partial x^I}$ so that the restriction of $\frac{\partial^{i}f}{\partial x^I}$
on the face $T_{I^c}^{d-k+i+1}$ is zero; consequently, $[f,g]_2 = 0$. For $1 \le i \le d-k+1$, 
the variables $x_{d-k+2},\dots,x_d$ appear as product factors in $\frac{\partial f}{\partial x_i}$, 
so that $\frac{\partial f}{\partial x_i}$ vanishes on the face $T_{\mb_k}^{d-k+1}$; consequently
$[f,g]_3=0$. Furthermore, since $f$ vanishes whenever one of $x_{d-k+2},\dots,x_d, x_{d+1}$ is zero, 
it follows that $[f,g]_4=0$. Thus, it remains to consider $[f,g]_1$. From the Rodrigue formula 
\eqref{Rodrigue} of $P_{\nb}^{(\g,\one_k)}(x)$, it follows that
$$
f(x) = \xb_{d-k+1}^{-\g}\frac{\partial^{|\nb|}} {\partial x^{\nb}}\left\{x^{(\g,\one_{k-1})+\nb}(1-|x|)^{|\nb|+1}\right\}.
$$
Taking derivative of $f$ with respect to variables whose index is in $\mb_k$, we obtain 
\begin{align*}
\frac{\partial^{k-1}f}{\partial x^{\mb_k}}= \frac{\partial^{k-1}f}{\partial x_{d-k+2}\cdots x_d} =&\,
 \xb_{d-k+1}^{-\g}\frac{\partial^{n-1}}{\partial x^{\nb+(\zero,\one_{k-1})}}\left\{x^{(\g,\one_{k-1})+\nb}(1-|x|)^{|\nb|+1}\right\} \\
=&\ (1-|x|)^{-(k-2)}P_{\nb+(\zero,\one_{k-1})}^{(\g,\zero_{k-1},-(k-2))}(x),
\end{align*}
from which we obtain immediately that 
\begin{align*}
   [f,g]_1 &\, =  \int_{T^d} \frac{\partial^{k-1}f}{\partial x^{\mb_k}} \frac{\partial^{k-1}g}{\partial x^{\mb_k}} (1- |x|)^{k -2} \xb_{d+1-k}^\g dx\\
&\,= \int_{T^d} \xb_{d+1-k}^\g P_{\nb+(\zero,\one_{k-1})}^{(\g,\zero_{k-1},-(k-2))}(x) \frac{\partial^{k-1}g} {\partial x^{\mb_k}}dx.
\end{align*}
Since $P_{\nb+(\zero,\one_{k-1})}^{(\g,\zero_{k-1},-(k-2))}(x)$  is a polynomial of degree $(n-1)$
and 
$$
    n-1 > (k-2) + n-k = (k-2)+ \deg \ \frac{\partial^{k-1}g} {\partial x^{\mb_k}} 
$$ 
for $g\in \Pi_{n-1}^{d}$, it follows from Lemma \ref{orthogo.} that $[f,g]_1=0$. This completes the
proof in this case. 

\medskip \noindent
{\it Case 2.2}. $f$ is an element of $x_{j,k}(\s) H_{n-j}^d\left(W_{\g, (\zero_{k-j},\one_j)\s}\right)$, 
where $1 \le j \le k-1$ and $\s \in \mathcal{G}_k$, and $x_{j,k}$ contains the variable $x_{d+1}$. 

First of all, since the inner product ${\la f,g\ra}_{\g,\one_k}$ involves symmetric derivatives 
with respect to variables $\{x_{d-k+2},\dots,x_d\}$, it is sufficient to consider the case that 
$\s = id$. Let $j$ be fixed,  $1\le j \le k-1$. Since $x_{j,k}$ contains the variable $x_{d+1}$, 
we can assume then
$$
 x_{j,k}=x_{d+2-j}\cdots x_d x_{d+1} \quad \hbox{and} \quad f(x) = x_{d+2-j}\cdots x_d x_{d+1}Q_{n-j}(x)
$$
where $Q_{n-j} \in H_{n-j}^d (W_{\g,\zero_{k-j},\one_j})$. From Definition \ref{defn:H_n^d} with $S_j$ 
being the set $S_{k,j}  =\{d-k+2,\dots,d+1-j \}$, we have then
\begin{align*}
  H_{n-j}^d (W_{\g,\zero_{k-j},\one_j})= \{P_\nb^{(\g,\zero_{k-j},\one_j)}(x):   \nb \in \NN_0^d, \,   
   |\nb|=n-j \quad \hbox{with} \, n_i=0, \, i \in S_{k,j}  \big\}.
\end{align*}

To keep the notation compact, we introduce the following notations in the proof. Recall
that for $x \in \RR^d$, $\nb\in\NN_0^d$ and $1 \le i \le d$, $\xb_i = (x_1,\ldots,x_i)$ and $\nb_i=(n_1,\ldots,n_i)$. For $1 \le i \le d$, we set  
$$
      \overleftarrow{\xb}_{i}:=(x_i,\ldots, x_d)\quad\textit{and}\quad \overleftarrow{\nb}_{i}:=(n_i,\ldots, n_d). 
$$
Now, for $1 \le j \le k-1$, we introduce the notation 
\begin{align*}
 \bg_{k,j}:=(\g, \zero_{k-j}, \one_{j-1})\quad\hbox{and}\quad & 
  \nb_{k,j}: =(\nb_{d-k+1},\zero_{k-j}, \overleftarrow{\nb}_{d+2-j}).
\end{align*}
Using these notations, we can write the element $Q_{n-j}  \in H_{n-j}^d(W_{\g,\zero_{k-j},\one_j})$ as
\begin{align*}
Q_{n-j}(x)= x^{-\bg_{k,j}} (1-|x|)^{-1}\frac{\partial^{n-j}} {\partial x^{\nb_{k,j}} } 
       \left\{x^{\bg_{k,j} + \nb_{k,j}} (1-|x|)^{|\nb_{k,j}|+1} \right\},
\end{align*}
where $|\nb_{k,j}|=n-j$. Consequently, our $f$ is of the form 
\begin {align}\label{equation 4}
  f(x)=\xb_{d-k+1}^{-\g}\frac{\partial^{n-j}} {\partial x^{\nb_{k,j}}} 
       \left\{x^{\bg_{k,j} + \nb_{k,j}} (1-|x|)^{|\nb_{k,j}|+1} \right\}.
\end{align}

The expression shows that $f(x) =0$ whenever one of $x_{d+2-j},\dots,x_d, x_{d+1}$ is zero. 
It follows readily that $[f,g]_4=0$, since the indices of $\{ x_{d+2-j},\dots,x_d, x_{d+1}\}$ are 
in $\mb_k^+$ for $1\le j \le k-1$. The variables in $\{x_{d+2-j},\dots,x_d\}$ whose indices are 
in $\mb_k$ appear as product factors in $\frac{\partial f}{\partial x_i}$, so that 
$\frac{\partial f}{\partial x_i}$ vanishes on the face $T_{\mb_k}^{d-k+1}$ except when $j=1$; 
consequently $[f,g]_3=0$ except when $j=1$. 
If $j=1$ then $f(x)=x_{d+1}Q_{n-1}(x)$, where $Q_{n-1}(x) \in H_{n-1}^d (W_{\g,\zero_{k-1},1})$,
so that an integration by parts shows that
\begin{align*}
 [f,g]_3= -\sum_{i=1}^{d-k+1} {\l_i} \int_{T_{\mb_k}^{d-k+1}} \xb_{d+1-k}^\g  f(x)\left\{(\g_i+1)\frac{\partial g}{\partial x_i}+x_i \frac{\partial^2 g}{\partial {x_i}^2}\right\}d\xb_{d-k+1}.
\end{align*}
Since $f=(1-|x|)Q_{n-1}(x)$ and  from Lemma \ref{lem:H_n^d},  $Q_{n-1} \big\vert_{T_{\mb_k}^{d-k+1}} 
\in \CV_{n-1}^{d-k+1} (W_{\g,1})$, it follows that $[f,g]_3=0$ in the case of $j=1$ as the term 
in the curly bracket is a polynomial of degree at most $n-2$. Thus, $[f,g]_3=0$ for $1\le j\le k-1$.

Next we consider $[f,g]_1$. For convenience, we introduce the following 
notation: For a fixed $j$ with $1 \le j \le k-1$, define 
$$
  \mb_{k,j}:=\{d-k+2,\ldots,d-j+1\} \quad \hbox{and} \quad \mb_{k,j}^c:=\{d-j+2,\ldots,d\}
$$
so that $\mb_k =\mb_{k,j}\cup \mb_{k,j}^c$. In particular, if $j=1$, then $\mb_{k,j} =\mb_k$ 
and $\mb_{k,j}^c=\emptyset$. 
If $n < k$, then $|\mb_k| = k-1>n-1$ and $\frac{\partial^{k-1}g}{\partial x^{\mb_k}}=0$, 
so that $[f,g]_1=0$. If $n \ge k$, then $|\mb_k|\le n-1$. Assume first $2 \le j\le k-1$. Then
$|\mb_{k,j}^c| \ge 1$. We perform 
integration by parts $|\mb_{k,j}^c|$ times, once for each of the variables ${x_{d+2-j},\ldots,x_d}$, 
which appear as product factors in $f$, to obtain
\begin{align*}
[f,g]_1  &\, = \int_{T^d}\frac{\partial^{k-1}f}{\partial x^{\mb_k}} \frac{\partial^{k-1}g}{\partial x^{\mb_k}} (1- |x|)^{|\mb_k|-1} \xb_{d+1-k}^\g dx\\
  &\, = (-1)^{|\mb_{k,j}^c|}\int_{T^d} \xb_{d+1-k}^\g \frac{\partial^{|\mb_{k,j}|}f}{\partial x^{\mb_{k,j}}}\frac{\partial^{|\mb_{k,j}^c|}}{\partial x^{\mb_{k,j}^c}}\left\{\frac{\partial^{k-1}g}{\partial x^{\mb_k}} (1- |x|)^{|\mb_k|-1}\right\} dx.
\end{align*}
Using the product rule, for a generic function $h$,
\begin{equation}\label{pro. rule}
\frac{\partial}{\partial x}\left\{h (x)(1-x)^n\right\}=(1-x)^{n-1}\left\{(1-x) h'(x)-n h(x)\right\},
\end{equation}
it is easy to see that we can write 
$$
\frac{\partial^{|\mb_{k,j}^c|}}{\partial x^{\mb_{k,j}^c}}\left\{\frac{\partial^{k-1}g}{\partial x^{\mb_k}} (1- |x|)^{|\mb_k|-1}\right\} =(1- |x|)^{|\mb_{k,j}|-1}h_{n-|\mb_k|-1}(x)
$$
where $h_{n-|\mb_k|-1}$ is a polynomial of degree ${(n-|\mb_k|-1)}$. Consequently, we have
$$
   [f,g]_1=(-1)^{|\mb_{k,j}^c|}\int_{T^d} \xb_{d+1-k}^\g 
\frac{\partial^{|\mb_{k,j}|}f}{\partial x^{\mb_{k,j}}}\left\{(1- |x|)^{|\mb_{k,j}|-1}h_{n-|\mb_k|-1}(x) \right\} dx.
$$
Now, taking derivative of $f$ in \eqref{equation 4} with respect to variables whose indices are 
in $\mb_{k,j}$, we obtain, as these derivatives only apply to the factor $(1-|x|)$ in $f$, that 
\begin{align}\label{equation 6}
 \frac{\partial^{|\mb_{k,j}|}f}{\partial \xb_{k,j}^{\mb_{k,j}}}=c_{k,j} x_{d+2-j}\cdots x_d 
    (1-|x|)^{-(|\mb_{k,j}|-1)} P^{(\g,\zero_{k-j},\one_{j-1},-(|\mb_{k,j}|-1))}_{\nb}(x),
\end{align}
where $c_{k,j}$ is a constant, $\xb_{k,j} = (x_{d-k+2},\ldots, x_{d-j+1})$ and $\nb = \nb_{k,j}$. 
Substituting it into the expression for $[f,g]_1$, we obtain
\begin{align*}
 [f,g]_1= & (-1)^{|\mb_{k,j}^c|} c_{k,j}\int_{T^d} \xb_{d+1-k}^\g x_{d+2-j}\cdots x_d \\
       & \times P^{(\g,\zero_{k-j},\one_{j-1},-(|\mb_{k,j}|-1))}_{\nb}(x)h_{n-|\mb_k|-1}(x)dx.
\end{align*}
Since $|\mb_{k,j}^c|=j-1$ and $|\nb|=n-j$, we can write $|\nb|=n-|\mb_{k,j}^c|-1$. As 
$|\mb_k|=|\mb_{k,j}|+|\mb_{k,j}^c|$, we see that 
$$
\deg  P_\nb^{(\g,\zero_{k-j},\one_{j-1},-(|\mb_{k,j}|-1))} > (|\mb_{k,j}|-1)+ \deg  h_{n-|\mb_k|-1},
$$ 
so that, by Lemma \ref{orthogo.}, $[f,g]_1=0$ for all $g\in \Pi_{n-1}^{d}$ when $2\le j\le k-1$. 
In the case $j=1$, $|\mb_{k,j}^c|=0$ and there is no need for integration by parts in $[f,g]_1$. 
In this case, \eqref{equation 6} becomes
\begin{align*}
 \frac{\partial^{|\mb_k|}f}{\partial x^{\mb_k}}=c_{k,1} (1-|x|)^
{-(|\mb_k|-1)} P^{(\g,\zero_{k-1},-(|\mb_k|-1))}_{\nb}(x)
\end{align*}
and $|\nb|=n-1$, so that $[f,g]_1$ becomes 
\begin{align*}
[f,g]_1&\,=\int_{T^d}\frac{\partial^{k-1}f}{\partial x^{\mb_k}} \frac{\partial^{k-1}g}{\partial x^{\mb_k}} (1- |x|)^{|\mb_k|-1} \xb_{d+1-k}^\g dx \\
&\, = c_{k,1}\int_{T^d} \xb_{d+1-k}^\g P_{\nb}^{(\g,\zero_{k-1},-(|\mb_k|-1))} (x)\frac{\partial^{k-1}g}{\partial x^{\mb_k}}dx. 
\end{align*}
Since $|\nb|>( |\mb_k|-1)+ \deg \ \frac{\partial^{k-1}g}{\partial x^{\mb_k}}$ for $g\in \Pi_{n-1}^{d}$, it 
follows from Lemma \ref{orthogo.} that $[f,g]_1=0$. This completes the proof that $[f,g]_1=0$. 

Next we deal with $[f,g]_2=0$. In the proof, let $I=I_1\cup I_2$, $I_1 \subset \mb_{k,j}$ and $I_2 \subset \mb_{k,j}^c$. Firstly, assume that $I_2$  is a proper subset of $\mb_{k,j}^c$ 
for $2\le j \le k-1$. Then at least one of the variables in $\{x_{d+2-j},\dots,x_d\}$ whose indices 
are in $\mb_{k,j}^c$ appears as a product factor in $\frac{\partial^{i}f}{\partial x^I}$, so that
the function $\frac{\partial^{i}f}{\partial x^I}$ vanishes on the face ${T_{I^c}^{d-k+i+1}}$ and,
consequently, $[f,g]_2=0$. Secondly, assume that $I_2 = \mb_{k,j}^c$. Then 
$|I_2|=|\mb_{k,j}^c|=j-1$ for $1 \le j\le k-1$ and the sum over $i$ in $[f,g]_2$ starts at 
$i = j-1$ if $j > 1$, so that we can write
\begin{align} \label{[f,g]_2}
[f,g]_2 =  \sum_{i= \max \{1,j-1\}}^{k-2} \sum_{\substack{ I_1\subset {\mb_{k,j}}  \\
 {I_2=\mb_{k,j}^c} \\ |I| =i}} {\l_I}  [f,g]_{2,i,I},
\end{align}    
where 
$$
[f,g]_{2,i,I}: = \int_{T_{I^c}^{d-k+i+1}} \frac{\partial^{i}f}{\partial x^{I}}\frac{\partial^{i}g}{\partial x^{I}}
      (1- |x|)^{|I|-1} \xb_{d+1-k}^\g d\xb_{d-k+1} d \xb_I. 
$$

For $j \ge 1$, $|I_2| = j-1 \ge 0$. If $j > 1$, we apply integration by parts $|I_2|$ times over 
the variables ${x_{d+2-j},\ldots,x_d}$, which appear as product factors in $f$ by 
\eqref{equation 4}, to obtain 
\begin{align*}
 [f,g]_{2,i, I}:= &\, \int_{T_{I^c}^{d-k+i+1}} \frac{\partial^{i}f}{\partial x^{I}}\frac{\partial^{i}g}{\partial x^{I}} 
  (1- |x|)^{|I|-1} \xb_{d+1-k}^\g d\xb_{d-k+1} d \xb_I\\
 = &\, (-1)^{|I_2|}\int_{T_{I^c}^{d-k+i+1}} \xb_{d+1-k}^\g \frac{\partial^{|I_1|}f}{\partial x^{I_1}}
    \frac{\partial^{|I_2|}}{\partial x^{I_2}}\left\{ \frac{\partial^{|I|}g}{\partial x^{I}}(1- |x|)^{|I|-1}\right\} 
    d\xb_{d-k+1} d \xb_I
\end{align*}
for $i = |I|\ge j -1$. If $I_1 \ne \emptyset$ then, using the product rule \eqref{pro. rule}, we can write
$$
\frac{\partial^{|I_2|}}{\partial x^{I_2}}\left\{ \frac{\partial^{|I|}g}{\partial x^{I}}(1- |x|)^{|I|-1}\right\}
   =(1- |x|)^{|I_1|-1}h_{n-|I|-1},
$$
where $h_{n-|I|-1}$ is a polynomial of degree $(n-|I|-1)$, we further deduce that 
\begin{align*}
 [f,g]_{2,i,I} =(-1)^{|I_2|}\int_{T_{I^c}^{d-k+i+1}} \xb_{d+1-k}^\g \frac{\partial^{|I_1|}f}{\partial x^{I_1}}\left\{(1- |x|)^{|I_1|-1}h_{n-|I|-1}\right\} d\xb_{d-k+1} d \xb_I,
\end{align*}
which also holds if $j =1$, for which there is no need to take derivatives as $|I_2| =0$. 
Furthermore, similar to \eqref{equation 6}, if we take derivatives of \eqref{equation 4} with 
respect to variables whose indices are in $I_1$, then the derivatives apply only on the factor
$(1-|x|)$, so that 
\begin{align}\label{equation 7}
  \frac{\partial^{|I_1|}f}{\partial x^{I_1}}= d_{k,j} x_{d+2-j}\cdots x_d(1-|x|)^{-(|I_1|-1)}
      P^{(\g,\zero_{k-j}, \one_{j-1},-(|I_1|-1))}_{\nb}(x),
\end{align}
where $\nb=\nb_{k,j}$ and $d_{k,j}$ is a constant, which allows us to write 
\begin{align*}
 [f,g]_{2,i,I} = &\,(-1)^{|I_2|}d_{k,j}\int_{T_{I^c}^{d-k+i+1}} \xb_{d+1-k}^\g x_{d+2-j}\cdots x_d  \\
    &\times P^{(\g,\zero_{k-j},\one_{j-1},-(|I_1|-1))}_{\nb} (x) h_{n-|I|-1} (x)d\xb_{d-k+1} d \xb_I.
\end{align*}
By  Lemma \ref{lem:H_n^d}, the restriction of $P^{(\g,\zero_{k-j},\one_{j-1},-(|I_1|-1))}_{\nb} (x)$ 
on the face ${T_{I^c}^{d-k+i+1}}$ is the polynomial 
$P^{(\g,\zero_{i-j+1},\one_{j-1},-(|I_1|-1))}_{\nb}(\xb_{d-k+1},\xb_I)$ in the $(d-k+i+1)$-variables 
of $\xb_{d-k+1}$ and $\xb_I$ on ${T_{I^c}^{d-k+i+1}}$. Since $|I_2| = j-1$ 
and $|\nb| = n-j$, we have $|\nb|> (|I_1|-1)+ \deg h$ and, consequently, by Lemma \ref{orthogo.}, 
that $[f,g]_{2,i,I} =0$. As a result, we have shown that the first term of $[f,g]_2$ in 
\eqref{[f,g]_2} is zero for $1 \le j \le k-2$ and $I_1 \ne \emptyset$.  

If, however, $I_1 = \emptyset$, then $I=I_2 = \mb_{k,j}^c$, so that $i = |I| = j -1$ and, in particular,
as $i \ge 1$, $j \ge 2$. We apply integration by parts $(|I|-1)$ times over the variables 
${x_{d+3-j},\ldots,x_d}$ to obtain
\begin{align*}
 [f,g]_{2,j-1,I} =  
  (-1)^{|I|-1}\int_{T_{I^c}^{d-(k-j)}} \xb_{d+1-k}^\g \frac{\partial f}{\partial x_{d+2-j}}h_{n-|I|-1}
      d\xb_{d-k+1} d\xb_I,
\end{align*}
where $h_{n-|I|-1}$ is the polynomial of degree $(n-|I|-1)$ given by
$$
h_{n-|I |-1}=\frac{\partial^{|I|-1}}{\partial x_{d+3-j}\cdots \partial x_d}
     \left\{ \frac{\partial^{|I|}g}{\partial x^{I}}(1- |x|)^{|I|-1}\right\}.
$$
Now, if $j \ge 3$, taking derivative of \eqref{equation 4} with respect to $x_{d+2-j}$ gives 
\begin{align}\label{equation 9}
\frac{\partial f}{\partial x_{d+2-j}}=x_{d+3-j}\cdots x_d P_{\nb}^{(\g,\zero_{k-j+1},\one_{j-2},0)}(x),
\end{align}
where $\nb=\nb_{k,j}+e_{d+2-j}$ and $|\nb|=n-|I|$, from which follows immediately that 
\begin{align*}
 [f,g]_{2,j-1,I} =  & \, (-1)^{|I|-1}\int_{T_{I^c}^{d-(k-j)}} \xb_{d+1-k}^\g x_{d+3-j}\cdots x_d  \\
   & \times P_{\nb}^{(\g,\zero_{k-j+1},\one_{j-2},0)} (x)h_{n-|I|-1} (x)    d\xb_{d-k+1}d \xb_I. 
\end{align*}
Since $|\nb| = n -|I|$ and, by Lemma \ref{lem:H_n^d}, $P_{\nb}^{(\g,\zero_{k-j+1},\one_{j-2},0)}(x) 
\big\vert_{T_{I^c}^{d-(k-j)}}$ is an element of $\CV_{n-|I|}^{d-(k-j)}(W_{\g,0,\one_{j-2},0})$ in
the variables $(\xb_{d-k+1},\xb_I)$, it follows from the orthogonality of 
$\CV_{n-|I|}^{d-(k-j)}(W_{\g,0,\one_{j-2},0})$ that $[f,g]_{2,j-1,I} = 0$ for $j \ge 3$. If, however,
$j=2$, then $I_2=\{d\}$ and  $[f,g]_{2,j-1,I}$ for $j =2$ becomes 
\begin{align*}
 [f,g]_{2,1,I} = \int_{T_{\{d-k+2,\dots,d-1\}}^{d-k+2}} \xb_{d+1-k}^\g
     \frac{\partial f}{\partial x_d} \frac{\partial g}{\partial x_d}d\xb_{d-k+1} dx_d.
\end{align*}
By \eqref{equation 9} with $j=2$, $\frac{\partial f}{\partial x_{d}}=P_{\nb}^{(\g,\zero_{k-1},0)}(x)$, 
which, by Lemma \ref{lem:H_n^d}, is an element in $\CV_{n-1}^{d-k+2}(W_{\g,0,0})$ when 
restricted to the face $T_{\{d-k+2,\dots,d-1\}}^{d-k+2}$. Consequently, since $|\nb | =n-1$
in this case, it follows from the orthogonality of $\CV_{n-1}^{d-k+2}(W_{\g,0,0})$ that
$ [f,g]_{2,1,I} =0$. Thus, we have establish that $[f,g]_{2,i,I} =0$ whenever $j-1 \le i \le k-2$ 
and $j \ge 2$. 

Putting these together, we have proved that $\la f,g\ra_{\g, - \one_k} =0$ for all $g \in \Pi_{n-1}^d$. Consequently, the proof of Case 2.2 is completed. 

\medskip \noindent
{\it Case 2.3}. $f$ is an element of $x_{j,k}(\s) H_{n-j}^d\left(W_{\g, (\zero_{k-j},\one_j)\s}\right)$, 
where $1 \le j \le k-1$ and $\s \in \mathcal{G}_k$, and $x_{j,k}$ does not contain the variable $x_{d+1}$. 

As in the Case 2.2, it is sufficient to consider $\s = id$. Since $x_{j,k}$ does not contain $x_{d+1}$,
we can assume that 
$$
   x_{j,k}=x_{d+2-j}\cdots x_d \quad\hbox{and} \quad f(x)=x_{d+2-j}\cdots x_d Q_{n-j+1}(x)
$$ 
for $2\le j \le k$, where $Q_{n-j+1} \in H_{n-j+1}^d (W_{\g,\zero_{k-j},\one_{j-1},0})$. 
From Definition \ref{defn:H_n^d}, $x_{S_{k-j+1}}=\{1-|x|,x_{d-k+2},\ldots,x_{d-j+1}\}$ and then
\begin{align*}
H_{n}^d (W_{\g,\zero_{k-j},\one_{j-1},0}) = &
  \left\{P_{\nb}^{(\zero_{k-j+1},\one_{j-1},\g)}(x_{S_{k-j+1}}, \overleftarrow{\xb}_{d-j+2},\xb_{d-k}): \right. \\
    & \left. \quad \nb=(\zero,{\nb}_{k,j}), \, |\nb| = n, \, \nb_{k,j}\in \NN_0^{d-k+j-1}  \right \}.
\end{align*}
 From Definition \ref{defn: P(x_S)}, we can explicitly write the Rodrigue basis of this space. To keep the notation compact, we use following notations in this case of the proof, 
\begin{align*} 
\xb_{k,j}:=  (\overleftarrow{\xb}_{d-j+2},\xb_{d-k}) \quad \hbox{and}
     \quad  \bg_{k,j}:=(\one_{j-1},\bg_{d-k}).
\end{align*} 
In case 2 of Definition \ref{defn: P(x_S)}, taking
\begin{align*}
&  x_{i_2} =x_{d-k+2},\ldots, x_{i_k}=x_d, \quad x_{i_{k+1}}=x_1,\dots, x_{i_{d+1}}=x_{d-k+1} \\
& \nb=(\nb_{k-j+1}, \nb_{k,j})\quad \textit {with} \quad \nb_{k-j+1}=(n_1,\ldots,n_{k-j+1}), 
 \quad \nb_{k,j}\in \NN_0^{d-k+j-1},
\end{align*}
then we can write 
\begin{align*}
 & P_{\nb}^{(\zero_{k-j+1},\bg_{k,j},\g_{d-k+1})}(x_{S_{k-j+1}},\xb_{k,j})   \\
  & \qquad =(-1)^{n_1}\xb_{k,j}^{- \bg_{k,j}} x_{d-k+1}^{-\g_{d-k+1}} 
     \frac{\partial^{|\nb|}}{\partial \ub_{\kb,\jb}^\nb} \left \{
     x_{S_{k-j+1}}^{\nb_{k-j+1}}\xb_{k,j}^{ \bg_{k,j} + \nb_{k,j}} x_{d-k+1}^{\g_{d-k+1}+|\nb|}\right\},
\end{align*}
where $\partial \ub_{\kb,\jb} = (\partial_{d-k+1}, \partial_{d-k+2,d-k+1},\dots,\partial_{d,d-k+1},
 \partial_{1,d-k+1}, \ldots, \partial_{d-k,d-k+1})$ and recall that 
$\partial_{i,j}=\partial_i-\partial_j$. Since the first $(k-j+1)$ components of $\nb$ are zero, 
 in the definition of $H_{n}^d (W_{\g,\zero_{k-j},\one_{j-1},0})$, we can write $Q_{n-j+1} \in 
 H_{n-j+1}^d (W_{\g,\zero_{k-j},\one_{j-1},0})$ as 
\begin{align*}
 Q_{n-j+1}(x)= \xb_{k,j}^{-\bg_{k,j}} x_{d-k+1}^{-\g_{d-k+1}} 
 \frac{\partial^{|\nb_{k,j}|}}{\partial \yb_{\kb,\jb}^{\nb_{k,j}}} \left \{
     \xb_{k,j}^{ \bg_{k,j} + \nb_{k,j}} x_{d-k+1}^{\g_{d-k+1}+|\nb_{k,j}|}  \right\},
\end{align*}
where $\partial \yb_{\kb,\jb} = (\partial_{d-j+2,d-k+1},\dots,\partial_{d,d-k+1},
 \partial_{1,d-k+1}, \ldots, \partial_{d-k,d-k+1})$. 
Hence, we conclude that   
\begin{align}\label{equation 11}
  f(x)=\xb_{d-k+1}^{-\bg_{d-k+1}} \frac{\partial^{n-j+1}} {\partial \yb_{\kb,\jb}^{\nb_{k,j}}} 
    \left\{\xb_{k,j}^{\bg_{k,j}+\nb_{k,j}} x_{d-k+1}^{|\nb_{k,j}|+\g_{d-k+1}} \right\},
\end{align}
where $|\nb_{k,j} | = n-j+1$. 

This expression of $f$ shows that $f(x) =0$ whenever one of $x_{d+2-j},\dots,x_d$ is zero. 
It follows readily that $[f,g]_4=0$, since  the indices of $\{ x_{d+2-j},\dots,x_d\}$ are among the variables whose indices are in $\mb_k^+$ for $2\le j \le k$, so that the restriction of $f$ on the face $T^{d-k}_{\mb_k^+}$ 
is zero. Furthermore, $[f,g]_3=0$ for the same reason. Since the variables in 
$\{x_{d-k+2},\ldots,x_{d-j+1}\}$, $2\le j \le k-1$, do not appear in $f$, 
$\frac{\partial^{|\mb_{k}|}f}{\partial x^{\mb_{k}}}=0$, so that $[f,g]_1=0$ except for $j=k$. We 
now consider the case of $j=k$. In this case,  by \eqref{equation 11} with $j=k$, 
\begin{align}\label{equ. j=k}
f(x)=\xb_{d-k+1}^{-\bg_{d-k+1}} \frac{\partial^{n-k+1}} {\partial \yb_{\kb,\kb}^{\nb_{k,k}}} 
    \left\{\xb_{k,k}^{\bg_{k,k}+ \nb_{k,k}} x_{d-k+1}^{|\nb_{k,k}|+\g_{d-k+1}} \right\},
\end{align}
where $|\nb_{k,k}|=n-k+1$.
Applying integration by parts $|\mb_k|-1$ times over the variables ${x_{d-k+3},\ldots,x_d}$,
which are product factors in $f$, we obtain
\begin{align*}
[f,g]_1  =(-1)^{|\mb_k|-1} \int_{T^d} \xb_{d+1-k}^\g \frac{\partial f}{\partial x_{d-k+2}} 
       \frac{\partial^{|\mb_k|-1}} {\partial \xb_{d-k+3}}\left\{\frac{\partial^{d-k}g}{\partial x^{\mb_k}} 
      (1- |x|)^{|\mb_k|-1}\right\} dx
\end{align*}
where $\frac{\partial^{|\mb_k|-1}} {\partial \xb_{d-k+3}}=\frac{\partial^{|\mb_k|-1}} {\partial x_{d-k+3}\cdots \partial x_d}$.
Taking derivative of $f$ in \eqref{equ. j=k} with respect to $x_{d-k+2}$, we obtain 
$$
\frac{\partial f}{\partial x_{d-k+2}}= c\, x_{d-k+3}\cdots x_d P^{(0,0,\one_{k-2},\g)}_{\nb}
   (x_{S_{k-j+1}},\xb_{k,j}),
$$
where $c$ is a constant and $\nb=(\zero,\nb_{k,k})$. Since $\frac{\partial^{|\mb_k|-1}} {\partial \xb_{d-k+3}}\left\{\frac{\partial^{d-k}g}{\partial x^{\mb_k}}
(1- |x|)^{|\mb_k|-1}\right\} $ is a polynomial of degree $n-|\mb_k|-1$ and  $|\nb|= n-k+1 = n-|\mb_k|$,
we deduce that $[f,g]_1=0$ form the fact that $P^{(0,0,\one_{k-2},\g)}_{\nb} \in 
 \CV_{n-|\mb_k|}^{d}(W_{(\g,0,\one_{k-2},0)})$. Thus, $[f,g]_1=0$ for $2\le j\le k$.

It remains to consider $[f,g]_2$. Firstly, assume that $I_2$ is a proper subset of $\mb_{k,j}^c$ 
for $2\le j \le k$.  Since at least one of the variables in $\{x_{d+2-j},\dots,x_d\}$ whose indices 
are in $\mb_{k,j}^c$ appears as a product factor in $\frac{\partial^{i}f}{\partial x^I}$, the function 
$\frac{\partial^{i}f}{\partial x^I}$ vanishes on the face ${T_{I^c}^{d-k+i+1}}$ so that $[f,g]_2=0$. 
Secondly, assume that $I_2=\mb_{k,j}^c=\{d-j+2,\ldots,d\}$. Then $|I_2|=|\mb_{k,j}^c|=j-1$ 
for $2\le j\le k-1$.  We do not need to consider the case $j=k$ when $I_2=\mb_{k,j}^c$, since
if $j=k$ then $\mb_{k,j}=\emptyset$ and $\mb_k=\mb_{k,j}^c=I_2$, so that $I=\mb_k$, which
contradicts to the fact that $I$ is a proper subset of $\mb_k$ in $[f,g]_2$. Thus, we only need
to consider $2\le j\le k-1$, for which the summation over $i$ in $[f,g]_2$ starts at $i = j-1$,
so that we can write 
\begin{align} \label{[f,g]_2-2.3}
[f,g]_2 =  \sum_{i= j-1}^{k-2} \sum_{\substack{ I_1\subset {\mb_{k,j}}  \\
 {I_2=\mb_{k,j}^c} \\ |I| =i}} {\l_I}  [f,g]_{2,i,I},
\end{align}    
where $[f,g]_{2,i,I}$ are given as in \eqref{[f,g]_2}. Since the variables $x_{d-k+2},\ldots,x_{d-j+1}$ 
whose indices are in $\mb_{k,j}$ for $2\le j \le k-1$ do not appear in $f$, 
$\frac{\partial^{|I_1|}f}{\partial x^{I_1}}=0$ and, consequently, $[f,g]_{2,i,I} = 0$ if $i \ge j$. 
For $i = j-1$, $|I_2 | = |I| = j-1$ and $I_2 = I = \{d-j+2,\ldots,d\}$; in particular, 
$I_2^c = I^c = \{d-k+2,\ldots,d-j+1\}$. 
Applying integration by parts $|I_2|-1$ times over the variables ${x_{d+3-j},\ldots,x_d}$,
which appeared as product factors in $f$, we find 
\begin{align*}
[f,g]_{2,j-2,I} =(-1)^{j}\int_{T_{I^c}} \xb_{d+1-k}^\g \frac{\partial f}{\partial x_{d+2-j}}h_{n-j}(x)\
   d\xb_{d-k+1} d \xb_I,
\end{align*}
where $h_{n-j}$ is a polynomial of degree $n-j$ given by
$$
h_{n-j}(x)=\frac{\partial^{j-2}}{\partial \xb_{d+3-j}}\left\{ \frac{\partial^{j-1}g}{\partial x^{I_2}} (1- |x|)^{j-2}\right\}. 
$$
Now, taking derivative of $f$ in \eqref{equation 11} with respect to $x_{d+2-j}$, we obtain
$$
\frac{\partial f}{\partial x_{d+2-j}}= d\, x_{d-j+3}\cdots x_d 
    P^{(\zero_{k-j+2},\one_{j-2},\g)}_{\nb}(x_{S_{k-j+1}},\xb_{k,j}),
$$
where $d$ is a constant and $\nb=(\zero,\nb_{k,j})$. The polynomial $P^{(\zero_{k-j+2},\one_{j-2},\g)}_{\nb}(x_{S_{k-j+1}},\xb_{k,j})$ is, by Lemma \ref{lem:H_n^d}, an element of 
$\CV_{n-j+1}^{d-(k-j)}(W_{\g,0,\one_{j-2},0})$ when restricted to the face 
$T_{\{d-k+2,\dots,d-j+1\}}^{d-(k-j)}$. Consequently, since $|\nb|=n-j+1$ and $h$ is a polynomial
of degree $n-j$, $[f,g]_{2,j-1,I} =0$ follows from the orthogonality of 
$\CV_{n-j+1}^{d-(k-j)}(W_{\g,0,\one_{j-2},0})$. Hence, $[f,g]_2 =0$ for all $g \in \Pi_{n-1}^d$. 

Putting these together, we have proved that $\la f,g\ra_{\g, - \one_k} =0$ for all $g \in \Pi_{n-1}^d$. Consequently, the proof of Case 2.3 is completed. 

\medskip \noindent
{\it Case 2.4}. $f \in H_n^d (W_{\g,\zero_k})$. From Definition \ref{defn:H_n^d}, we can assume 
\begin{align*}
  f(x) = P_{\nb}^{(\zero,\g)}(x_{d-k+2},\ldots,x_{d+1},\xb_{d-k}),\quad |\nb|=n,
\end{align*}
where $x_{d+1}=1-|x|$ and $\nb=(\zero,\nb_{d-k})$, $\nb_{d-k}=(n_1,\dots,n_{d-k})$. Since 
the first $k$ components of $\nb$ are zero, the Rodrigue formula \eqref{Rodrigue} shows
that the variables whose indices are in $\mb_k$ do not really appear in $f$. Hence, 
$\frac{\partial^{k-1}f}{\partial x^{\mb_k}}=0$ and $\frac{\partial^{i}f}{\partial x^{I}}=0$ 
for $I \subset \mb_k$, from which it follows that $[f,g]_1=0$ and $[f,g]_2=0$, respectively. 
For the integral in $[f,g]_3$, we apply integration by parts to obtain 
\begin{align*}
 [f,g]_3=  -\sum_{j=1}^{d-k+1} {\l_j} \int_{T_{\mb_k}^{d-k+1}} \xb_{d+1-k}^\g  f\left\{(\g_j+1)\frac{\partial g}{\partial x_j}+x_j \frac{\partial^2 g}{\partial {x_j}^2}\right\}d\xb_{d-k+1}.
\end{align*}
Since, by Lemma \ref{lem:H_n^d}, $f \big \vert_{T_{\mb_k}^{d-k+1}} \in \CV_n^{d-k+1} (W_{\g,0})$,
and the term in the curly bracket is a polynomial of degree at most $n-2$, we conclude that 
$[f,g]_3=0$. Finally, by Lemma \ref{lem:H_n^d}, $f \big \vert_{T_{\mb_k^+}^{d-k}} \in 
\CV_n^{d-k} (W_{\g})$, which implies immediately that $[f,g]_4=0$. Thus, it follows from 
\eqref{ipd-decomp} that $ \la f, g\ra_{\g, -\one_k} =0$ for all $g\in \Pi_{n-1}^{d}$.

\medskip
Putting the four cases together, the decomposition \eqref{V-decom} shows that 
$\la f, g\ra_{\g, -\one_k} =0$ for every $f\in \CU_{n}^{d}(W_{\g,-\one_k})$ and all 
$g\in \Pi_{n-1}^{d}$. This completes the proof of (ii) of the theorem for {\bf Case 2}.

\medskip\noindent 
{\bf Case 3}. $k = d+1$. The main terms of the inner product $ \la f, g \ra_{ -\one}$
are essentially the case of $k = d+1$ of $[f,g]_1$ and $[f,g]_2$ in {\bf Case 2}, hence, as
shown in that case, that these two terms are zero for every $f\in \CU_{n}^{d}(W_{-\one})$ 
and for all $g\in \Pi_{n-1}^{d}$. Furthermore, if $f \in x_1 \cdots x_{d+1}\CV_{n-d-1}^d(W_\one)$ 
then $f$ vanishes on  $e_0, e_1, \ldots, e_d$, whereas if $f \in  x_{j,d+1}(\s) 
H_{n-j}^d (W_{(\zero_{d+1-j},\one)\s})$, then $f(e_0)=f(e_1)=\ldots=f(e_d)=0$. Thus, every 
$f\in \CU_{n}^{d}(W_{-\one})$ vanishes at all vertices, so that the last term in $\la f,g \ra_{-\one}$
is zero. Consequently, $\la f, g\ra_{ -\one} =0$ for $\forall f\in \CU_{n}^{d}(W_{-\one})$ and 
$\forall g\in \Pi_{n-1}^{d}$. 

The proof of (ii) in Theorem \ref{thm:main_d>2_2} is completed.  \qed

\subsection{Proof of (iii) in Theorem \ref{thm:main_d>2_2}}
We need to show that $\la f, g\ra_{\g,-1} =0$ for $f \in (1-|x|)  \CV_{n-1}^d(W_{\g,1})$ and
$g \in H_n^d(W_{\g,0})$. The proof uses a property of $H_n^d(W_{\g,0})$, which is of 
interesting in itself. 

By definition, $H_n^d(W_{\g,0}) = H_{n,\{d+1\}}^d(W_{\g,0})$, that is, $S_1=\{d+1\}$ in the item 
2 of Definition \ref{defn:H_n^d}, and $(\g_S, \g_{\ell_d}) = (0, \g', \g_{\ell_{d}})$,
where $\g' = \g_{S\setminus \{d+1\}} = (\g_{\ell_1},\ldots, \g_{\ell_{d-1}})$, and 
$x_S=(x_{S_1},x_{S\setminus S_1}) = (x_{d+1}, x')$, $x' =(x_{\ell_1} ,\ldots, x_{\ell_{d-1}})$. Hence,
by Definition \ref{defn: P(x_S)}, the elements of $H_n^d(W_{\g,0})$ are given by, since 
$n_1 =0$,  
$$
  P_\nb^{(\g_{S},\g_{\ell_d})}(x_S) = (x')^{-\g'}x_{\ell_{d}}^{-\g_{\ell_{d}}} 
      \frac{\partial^{|\nb|}}{ \partial \xb_{S}^{\nb'}}
      \left\{{(x')}^{\g'+\nb'} x_{\ell_{d}}^{\g_{\ell_{d}}+|\nb|}\right\},
$$
where $\nb = (0,\nb')$ and $\partial \xb_{S} =(\partial_{\ell_1, \ell_{d}}, \ldots, \partial_{\ell_{d-1}, \ell_{d}})$.
Consequently, the elements of $H_n^d(W_{\g,0})$ are homogeneous polynomials of degree
$|\nb|$, which satisfy, by Euler's formula, the equation $x_1 \partial_1 g + \ldots +
x_d \partial_d g = n g$. Taking derivative again shows easily that, if $g$ is a homogeneous 
polynomial of degree $n$, then 
$$
  \sum_{i=1}^d x_i^2  \partial_{i}^2 g  + 2 \sum_{1 \le i< j \le d} x_ix_j  \partial_i \partial_j g = n(n-1)g,     
$$
which applies, in particular, to $g \in H_n^d(W_{\g,0})$. On the other hand, by the item (i)
of Lemma \ref{lem:H_n^d} and Lemma \ref{rem}, $g \in H_n^d(W_{\g,0})$ satisfies the equation
$L_{\g,-1} g = -n(n+|\g| + d-1) g$, which is, by \eqref{diff-eqn} 
\begin{align*}
\sum_{i=1}^d x_i(1 -x_i) \partial_{i}^2 g & - 2 \sum_{1 \le i < j \le d} x_i x_j  \partial_i \partial_j g \\
   & + \sum_{i=1}^d \left( \g_i +1-(|\g|+d) x_i \right) \partial_i g  =  -n  \left(n+|\g| +d-1\right)g.
\end{align*}
Adding the two identities together and making use of the Euler's formula again, we conclude
that $g \in H_n^d(W_{\g,0})$ satisfies the following equation
\begin{equation} \label{last-eqn}
  \sum_{i=1}^d x_i \partial_{i}^2 g + \sum_{i=1}^d (\g_i + 1) \partial_i g  = 0.
\end{equation}

Now, let $f(x)=(1-|x|)P_{\nb}^{(\g,1)}(x)$, where $P_{\nb}^{(\g,1)}(x)\in\CV_{n-1}^d(W_{\g,1})$,
and $g \in H_n^d(W_{\g,0})$. Following the proof of Case 1 of the subsection 4.4, we then arrive
at \eqref{[f,g]-case1.1}. The summation over the derivatives of $g$, however, is exactly the 
differential operator in the left hand side of \eqref{last-eqn}, so that $\la f, g\ra_{\g,-1} =0$. 
\qed

\end{document}